\documentclass[11pt, requo, oneside]{amsart}

 \usepackage{graphicx,color,amssymb, amsmath, amsthm, mathrsfs}
\usepackage{geometry}
\usepackage[all]{xy}
\usepackage{slashed}
\usepackage{mathtools}
\usepackage{todonotes}
\usepackage{color}
\usepackage[colorlinks = true,
            linkcolor = blue,
            urlcolor  = blue,
            citecolor = red,
            anchorcolor = blue]{hyperref}

\swapnumbers
\newtheorem{theorem}{Theorem}[section]
\newtheorem{lemma}[theorem]{Lemma}
\newtheorem{proposition}[theorem]{Proposition}
\newtheorem{corollary}[theorem]{Corollary}
\theoremstyle{definition}
\newtheorem{definition}[theorem]{Definition}

\newtheorem{remark}[theorem]{Remark}

\newtheorem{example}[theorem]{Example}
\newtheorem*{remark*}{Remark}
\newtheorem*{definition*}{Definition}

\usepackage{thmtools}
\usepackage{amsmath,amscd}
\declaretheoremstyle[notefont=\bfseries,notebraces={}{},%
    headpunct={},postheadspace=1em]{mystyle}
\declaretheorem[style=mystyle,numbered=no,name=Theorem]{thm-hand}

\newcommand{\bea}          {\begin{eqnarray}}
\newcommand{\eea}          {\end{eqnarray}}
\newcommand{\beastar}          {\begin{eqnarray*}}
\newcommand{\eeastar}          {\end{eqnarray*}}

\newcommand{\supp}{\textnormal{supp }}

\begin{document}

\title{Persistence approximation property for $L^{p}$ operator algebras}

\author{Hang Wang}
\address[Hang Wang]{\normalfont{Research Center for Operator Algebras, School of Mathematical Sciences, East China Normal University, Shanghai 200062, China}}
\email{wanghang@math.ecnu.edu.cn}

\author{Yanru Wang}
\address[Yanru Wang]{\normalfont{Research Center for Operator Algebras, School of Mathematical Sciences\\ East China Normal University, Shanghai 200062, China}}
\email{52215500014@stu.ecnu.edu.cn}

\author{Jianguo Zhang}
\address[Jianguo Zhang]{\normalfont{School of Mathematics and Statistics, Shaanxi Normal University,
Xi'an 710119, China}}
\email{jgzhang@snnu.edu.cn}

\author{Dapeng Zhou}
\address[Dapeng Zhou]{\normalfont{School of Statistics and Information, Shanghai University of International Business and Economics, Shanghai 201620, China}}
\email{giantroczhou@126.com}

\begin{abstract} 
In this paper, we study the persistence approximation property for quantitative $K$-theory of filtered $L^p$ operator algebras. Moreover, we define quantitative assembly maps for $L^p$ operator algebras when $p\in [1,\infty)$. Finally, in the case of $L^{p}$ crossed products and $L^{p}$ Roe algebras, we find sufficient conditions for the persistence approximation property. This allows us to give some applications involving the $L^{p}$ (coarse) Baum-Connes conjecture.
\end{abstract}


\maketitle
\tableofcontents

\section{Introduction}
Quantitative operator $K$-theory was primarily developed first by Guoliang Yu in the article \cite{1998-Yu} on the Novikov conjecture for groups with finite asymptotic dimension, and then by Oyono-Oyono and Yu in \cite{Oyono-2015} to study a general quantitative $K$-theory for filtered $C^{*}$-algebras. Based on their work, Yeong Chyuan Chung later extended the framework of quantitative $K$-theory to the class of algebras of bounded linear operators on subquotients of $L^{p}$ spaces for $p\in [1,\infty)$ (i.e. $SQ_{p}$ algebras) in \cite{Banach-Chung}. Since an $L^{p}$ operator algebra is obviously an $SQ_{p}$ algebra, we can derive a framework of quantitative $K$-theory for $L^{p}$ operator algebras by applying Chung's work to the $L^{p}$ operator algebras. For a filtered $L^{p}$ operator algebra $A$, the $K$-theory of $A$ can be approximated by the quantitative $K$-theory group $K^{\varepsilon,r, N}_{*}(A)$ as $r$ and $N$ tend to infinity, i.e. $\lim\limits_{r,N\rightarrow\infty} K^{\varepsilon, r,N}_{*}(A)=K_{*}(A)$. Compared to the usual $K$-theory of a complex Banach algebra, quantitative $K$-theory is more computable and more flexible by using quasi-idempotents and quasi-invertibles instead of idempotents and invertibles respectively. 

To explore a way of approximating $K$-theory with quantitative $K$-theory, Oyono-Oyono and Yu studied the persistence approximation property for quantitative $K$-theory of filtered $C^{*}$-algebras in \cite{PAP-Oyono}. Subsequently, Qin Wang and Zhen Wang investigated the persistence approximation property for maximal Roe algebras. They proved that if $X$ is a coarsely uniformly contractible discrete metric space with bounded geometry, and it admits a fibred coarse embedding into Hilbert space, then the maximal Roe algebra for $X$ satisfies the persistence approximation property in \cite{WZ-PAP}. Motivated by these successful researches on the persistence approximation property for the quantitative $K$-theory, we will in this paper extend these methods and results for $C^{*}$-algebras to $L^{p}$ operator algebras. 

Recently, the research on $L^{p}$ operator algebras has been revived. In the work of  \cite{Phillips}, Phillips introduced full and reduced $L^{p}$ crossed products and proved that the $K$-theory of $L^{p}$ analogs of Cuntz algebras is the same as that of $C^{*}$-algebras. This work has inspired mathematicians to study $L^{p}$ operator algebras that behave like $C^{*}$-algebras, including group $L^{p}$ operator algebras \cite{Lp algebra-Gardella,Group algebras,$p$-convolution,Phillips} and groupoid $L^{p}$ operator algebras \cite{groupoids}. There is also related work on $\ell^p$ uniform Roe algebras in comparison with classical uniform Roe algebras, such as \cite{Rigidity, uniform, quasi-local}. These researches provide sufficient methods and techniques for dealing with the problem of the $L^p$ operator algebras in this paper.

In order to investigate an $L^{p}$ version of the persistence approximation property, we have to give a definition of the quantitative $L^{p}$ assembly map. In this important article \cite{DC-Chung}, Chung defined the $L^{p}$ assembly map, and showed that a certain $L^{p}$ assembly map is an isomorphism if the action $\Gamma\curvearrowright X$ has finite dynamical complexity.  Moreover, Jianguo Zhang and Dapeng Zhou in \cite{ZZ21} studied $L^{p}$ localization algebras and $L^{p}$ Roe algebras, which are basic ingredients for defining quantitative $L^{p}$ assembly maps.

The main aim of this paper is to define the $L^{p}$ analog of the quantitative assembly map to study the persistence approximation property for the quantitative $K$-theory of filtered $L^{p}$ operator algebras. More precisely, we say that a filtered $L^{p}$ operator algebra $A$ has the {\it persistence approximation property} if for any $\varepsilon$ in $(0,\frac{1}{20})$, any $r>0$ and any $N\geq 1$, there exist $\varepsilon'\in [\varepsilon,\frac{1}{20})$, $r'\geq r$ and $N'\geq N$ such that the following statement $\mathcal{PA_{*}}(A,\varepsilon,\varepsilon',r,r',N,N')$ is satisfied: an element from $K^{\varepsilon,r,N}_{*}(A)$ is zero in $K_{*}(A)$ implies that it is zero in $K^{\varepsilon',r',N'}_{*}(A)$. For the case of a crossed product of an $L^{p}$ operator algebra by a finitely generated group, we obtain the main theorem:

\begin{theorem}\rm(see Theorem \ref{theorem 4.1})\ \
Let $\Gamma$ be a finitely generated group, and let $A$ be a $\Gamma$-$L^{p}$ operator algebra. Assume that
\begin{itemize}
\item $\Gamma$ admits a cocompact universal example for proper actions;
\item for any positive integer $\mathscr{N}$, there exists a non-decreasing function $\omega: [1,\infty)\rightarrow [1,\infty)$ such that the $\mathscr{N}$-$L^{p}$ Baum-Connes assembly map for $\Gamma$ with coefficients in 
$$\ell^{\infty}(\mathbb{N}, \mathscr{K}(\ell^{p})\otimes A)$$
is $\omega$-surjective;
\item the $L^{p}$ Baum-Connes assembly map for $\Gamma$ with coefficients in $A$ is injective.
\end{itemize}
Then for any $N\geq 1$, there exists a universal constant $\lambda_{PA}\geq 1$ such that for any $\varepsilon$ in $(0,\frac{1}{20\lambda_{PA}})$ and any $r>0$, there exist $r'\geq r$ and $N'\geq N$ such that $\mathcal{PA_{*}}(A\rtimes\Gamma, \varepsilon,\lambda_{PA}\varepsilon,r,r',N,N')$ holds.
\end{theorem}

This theorem is a generalization of Oyono-Oyono and Yu's work on persistence approximation property for $C^{*}$ crossed products \cite{PAP-Oyono}. We call it the $L^{p}$ version of the persistence approximation property. To demonstrate this result, we define a quantitative $L^{p}$ assembly map by using the $L^{p}$ localization algebra and the $L^{p}$ Roe algebra. Moreover, we carefully estimate the changing parameters of $(\varepsilon,r,N)$-idempotent and $(\varepsilon,r,N)$-invertible elements in the proof of the theorem to present a cleaner result.

Parallel to the main theorem, we obtain a similar result for the $L^{p}$ Roe algebra for a discrete metric space $X$ with bounded geometry. Replacing the assumption that the group admits a cocompact universal example for proper actions by that $X$ is coarsely uniformly contractible, we have the following theorem: 

\begin{theorem}\rm (see Theorem \ref{th 5.6})\ \
Let $X$ be a discrete metric space with bounded geometry, and let $A$ be an $L^{p}$ operator algebra. Assume that
\begin{itemize}
\item $X$ is coarsely uniformly contractible;
\item for any positive integer $\mathscr{N}$, there exists a non-decreasing function $\omega: [1,\infty)\rightarrow [1,\infty)$ such that the $\mathscr{N}$-$L^{p}$ coarse Baum-Connes assembly map for $X$ with coefficients in  
$$\ell^{\infty}(\mathbb{N},\mathscr{K}(\ell^{p})\otimes A)$$
is $\omega$-surjective;
\item the $L^{p}$ coarse Baum-Connes assembly map for $X$ with coefficients in $A$ is injective.
\end{itemize}
Then for any $N\geq 1$, there exists a universal constant $\lambda_{PA}\geq 1$ such that for any $\varepsilon$ in $(0,\frac{1}{20\lambda_{PA}})$ and any $r>0$, there exist $r'\geq r$ and $N'\geq N$ such that $\mathcal{PA_{*}}(B^{p}(X,A), \varepsilon,\lambda_{PA}\varepsilon,r,r',N,N')$ holds.
\end{theorem}
As a corollary of this theorem, we proved that any $L^{p}$ Roe algebra for a discrete Gromov hyperbolic metric space satisfies the persistence approximation property.

The outline of this paper is as follows: In section 2, we recall the main results of quantitative $K$-theory for filtered $L^{p}$ operator algebras. In section 3, we define a quantitative $L^{p}$ assembly map and show the connection between the quantitative statements and the $L^{p}$ Baum-Connes conjecture. In section 4, for the case of $L^{p}$ crossed products, we find a sufficient condition for the persistence approximation property. Finally, in section 5, we show that if $X$ is a coarsely uniformly contractible discrete metric space with bounded geometry and finite asymptotic dimension, then the $L^{p}$ Roe algebra for $X$ has the persistence approximation property. 

\section{Quantitative $K$-theory for $L^{p}$ operator algebras}
The ordinary $K$-theory of Banach algebras developed in \cite{Blackadar} focuses on idempotents or invertibles. In comparison, quantitative $K$-theory for Banach algebras studied in \cite{Banach-Chung}  focuses on quasi-idempotents or quasi-invertibles. In this section, we recall some basic definitions and theorems of quantitative $K$-theory for filtered $SQ_{p}$ algebras from \cite{Banach-Chung}. Moreover, by applying these conclusions to filtered $L^p$ operator algebras, we can obtain some basic concepts and main results of quantitative $K$-theory for filtered $L^p$ operator algebras.
\begin{definition}\cite{Lp algebra-Gardella}
Let $A$ be a Banach algebra. For $p\in[1,\infty)$, we say that $A$ is an $L^p$ operator algebra if there exist an $L^p$ space $E$ and an isometric homomorphism $A\rightarrow\mathcal{B}(E)$.
\end{definition}

\begin{remark}
The $L^p$ operator algebra was initially defined by Phillips in \cite{Phillips}, and the
above definition is compatible with the original one.
\end{remark}

\begin{definition}\cite{Banach-Chung}
A filtered $L^{p}$ operator algebra is an $L^{p}$ operator algebra $A$ with a family $(A_r)_{r>0}$ of closed linear subspaces indexed by positive real numbers $r\in (0,\infty)$ such that 
\begin{itemize}
\item $A_r\subset A_{r'}$ if $r\leq r'$;

\item $A_{r}A_{r'}\subset A_{r+r'}$ for all $r,r'>0$;

\item the subalgebra $\bigcup\limits_{r>0}A_{r}$ is dense in $A$.
\end{itemize}
\end{definition}
If $A$ is unital with identity $1_{A}$, we require $1_{A}\in A_{r}$. For any $r>0$, we call the family $(A_{r})_{r>0}$ a filtration of $A$. We say that $a$ has propagation $r$ if $a\in A_{r}$.

If $A$ is not unital, we write the unitization of $A$ as 
$$A^{+}=\{(a,z): a\in A, z\in\mathbb{C}\}$$
with multiplication given by $(a,z)(a',z')=(aa'+za'+z'a,zz')$.
We use $\widetilde{A}$ to represent $A^{+}$ if $A$ is non-unital or to represent $A$ if $A$ is unital.

In order to control the matrix norm in quantitative $K$-theory of Banach algebras, we need to establish the matrix norm structure.
\begin{definition}\cite{daws}
 For $p\in[1,+\infty)$, an abstract $p$-operator space is a Banach space $X$ together with a family of norms $\Vert\cdot\Vert_n$ on $M_n(X)$ satisfying:
\begin{itemize}
\item $\mathcal{D}_\infty$: For $u\in M_n(X)$ and $v\in M_m(X)$, we have
$$\begin{Vmatrix}\begin{pmatrix} u&0\\0&v\end{pmatrix}\end{Vmatrix}_{n+m}=\mathrm{max}(\Vert u\Vert_n,\Vert v\Vert_m);$$
\item $\mathcal{M}_p$: For $u\in M_m(X)$, $\alpha\in M_{n,m}(\mathbb{C})$ and $\beta\in M_{m,n}(\mathbb{C})$, we have
$$\Vert\alpha u\beta\Vert_{n}\leq\Vert\alpha\Vert_{B\big(\bigoplus\limits_{i=1}^{m}\ell^{p},\bigoplus\limits_{i=1}^{n}\ell^{p}\big)}\Vert u\Vert_m\Vert \beta\Vert_{B\big(\bigoplus\limits_{i=1}^{n}\ell^{p},\bigoplus\limits_{i=1}^{m}\ell^{p}\big)}.$$
\end{itemize}
\end{definition}
Clearly, an $L^{p}$ operator algebra is an abstract $p$-operator space.
\begin{definition}\cite{Pisier}
Let $X$ and $Y$ be $p$-operator spaces, and let $\phi: X\rightarrow Y$ be a bounded linear map. For each $n\in\mathbb{N}$, let $\phi_n: M_n(X)\rightarrow M_n(Y)$ be the induced map given by $\phi_n([x_{ij}])=[\phi(x_{ij})]$. We say that $\phi$ is $p$-completely bounded if $\sup\limits_n\Vert\phi_n\Vert<\infty$. In this case, we let $\Vert\phi\Vert_{pcb}=\sup\limits_{n}\Vert\phi_n\Vert$.
\end{definition}
We say that $\phi$ is $p$-completely contractive if $\Vert\phi\Vert_{pcb}\leq 1$ and $\phi$ is $p$-completely isometric if $\Vert\phi\Vert_{pcb}=1$. 

\begin{definition}\cite{Banach-Chung}
Let $A$ and $B$ be filtered $L^{p}$ operator algebras with filtrations $(A_r)_{r>0}$ and $(B_r)_{r>0}$ respectively. A filtered homomorphism $\phi:A\rightarrow B$ is an algebra homomorphism such that 
\begin{itemize}
\item $\phi$ is $p$-completely bounded;

\item $\phi(A_r)\subset B_r$ for all $r>0$.
\end{itemize}
\end{definition}
If $\phi: A\rightarrow B$ is a filtered homomorphism, then it induces a filtered homomorphism $\phi^{+}: A^{+}\rightarrow B^{+}$ given by $\phi^{+}(a,z)=(\phi(a),z)$.
\begin{definition}\cite{Banach-Chung}
Let $A$ be a unital filtered $L^{p}$ operator algebra. For $0<\varepsilon<\frac{1}{20}$, $r>0$ and $N\geq 1$,
\begin{itemize}
\item an element $e\in A$ is called an $(\varepsilon,r,N)$-idempotent if $\Vert e^{2}-e\Vert<\varepsilon$, $e\in A_r$, and $\mathrm{max}(\Vert e\Vert,\Vert 1_{\tilde{A}}-e\Vert)\leq N$;
\item if $A$ is unital, an element $u\in A$ is called an $(\varepsilon,r,N)$-invertible if $u\in A_r$, $\Vert u\Vert\leq N$, and there exists $v\in A_r$ with $\Vert v\Vert\leq N$ such that $\mathrm{max}(\Vert uv-1 \Vert,\Vert vu-1 \Vert)<\varepsilon$.
\end{itemize}
\end{definition}
We call $v$ an $(\varepsilon,r,N)$-inverse for $u$ and we call $(u,v)$ an $(\varepsilon,r,N)$-inverse pair. In addition, $\varepsilon$ is called the control and $r$ is called the propogation of the $(\varepsilon,r,N)$-idempotent or of the $(\varepsilon,r,N)$-invertible.

Next, we recall the definitions of quantitative $K$-theory for filtered $L^{p}$ operator algebras. Given a filtered $L^{p}$ operator algebra $A$:
	\begin{itemize}
				\item we let $Idem^{\varepsilon,r,N}(A):=\{e\in A \mid e \text{ is an }(\varepsilon,r,N)$-idempotent\};
           \item we set $Idem^{\varepsilon,r,N}_{n}(A):=Idem^{\varepsilon,r,N}(M_{n}(A))$ for each $n\in\mathbb{N}$;
\item we have inclusions $Idem^{\varepsilon,r,N}_{n}(A)\hookrightarrow Idem^{\varepsilon,r,N}_{n+1}(A)  $, $e\mapsto\begin{pmatrix}e&0\\0&0\end{pmatrix}$;
				\item	 we set $Idem^{\varepsilon,r,N}_{\infty}(A):=\bigcup\limits_{n\in\mathbb{N}}Idem^{\varepsilon,r,N}_{n}(A)$;
          \item we define the equivalence relation $\sim$ on $Idem^{\varepsilon,r,N}_{\infty}(A)$ as follows: $e\sim f$ if and only if $e$ and $f$ are  $(4\varepsilon,r,4N)$-homotopic in $Idem^{4\varepsilon,r,4N}_{\infty}(A)$;
				\item	 we denote $[e]:=\{f\in Idem^{\varepsilon,r,N}_{\infty}(A)\mid f\sim e\text{ in }Idem^{\varepsilon,r,N}_{\infty}(A)\}$ ;
           \item $Idem^{\varepsilon,r,N}_{\infty}(A)/\sim:=\{[e] \mid e\in Idem^{\varepsilon,r,N}_{\infty}(A)\}$ and $[e]+[f]=\begin{bmatrix}\begin{pmatrix}e&0\\0&f\end{pmatrix}\end{bmatrix}$;
\item $Idem^{\varepsilon,r,N}_{\infty}(A)/\sim$ is an abelian semigroup with identity $[0]$.
			\end{itemize}
If we want to keep track of parameter changes, we write $[e]_{\varepsilon,r, N}$ instead of $[e]$.

\begin{definition}\cite{Banach-Chung}
Let $A$ be a filtered $L^{p}$ operator algebra. For $0<\varepsilon<\frac{1}{20}$, $r>0$ and $N\geq 1$,
\begin{itemize}
\item	 if $A$ is unital, define $K^{\varepsilon,r,N}_{0}(A)$ to be the Grothendieck group of $Idem^{\varepsilon,r,N}_{\infty}(A)/\sim$;
\item	 if $A$ is non-unital, define $K^{\varepsilon,r,N}_{0}(A):= ker\big(\pi_{*}: K^{\varepsilon,r,N}_{0}(A^+)\rightarrow K^{\varepsilon,r,N}_{0}(\mathbb{C})\big)$,
where $\pi:A^{+}\rightarrow\mathbb{C}$ is the usual quotient homomorphism, which is $p$-completely contractive.
\end{itemize}
\end{definition}
If $[e]-[f]\in K^{\varepsilon,r,N}_{0}(A)$, where $e,f\in M_{k}(\widetilde{A})$, then $[e]-[f]=[e']-[I_k]$ in $K_{0}^{\varepsilon,r,N}(A)$ for some $e'\in M_{2k}(\tilde{A})$. Therefore, if we relax control, we can write elements in $K^{\varepsilon,r,N}_{0}(A)$ in the form $[e]-[I_k]$ with $\pi(e)=diag(I_k,0)$.

Given a unital filtered $L^{p}$ operator algebra $A$:
	\begin{itemize}
				\item we let $GL^{\varepsilon,r,N}(A):=\{u\in A \mid u \text{ is an }(\varepsilon,r,N)$-invertible\};
				\item	 we set $GL^{\varepsilon,r,N}_{n}(A):=GL^{\varepsilon,r,N}(M_{n}(A))$ for each positive integer $n$;
\item we have inclusions $GL^{\varepsilon,r,N}_{n}(A)\hookrightarrow GL^{\varepsilon,r,N}_{n+1}(A)$, $ u\mapsto\begin{pmatrix}u&0\\0&1\end{pmatrix}$;
\item  we set $GL^{\varepsilon,r,N}_{\infty}(A):=\bigcup\limits_{n\in\mathbb{N}}GL^{\varepsilon,r,N}_{n}(A)$;
\item we define the equivalence relation $\sim$ on $GL^{\varepsilon,r,N}_{\infty}(A)$ as follows: $u\sim v$ if and only if $u$ and $v$ are  $(4\varepsilon,2r,4N)$-homotopic in $GL^{4\varepsilon,2r,4N}_{\infty}(A)$;
				\item	 we denote $[u]:=\{v\in GL^{\varepsilon,r,N}_{\infty}(A) \mid v\sim u\text{ in }GL^{\varepsilon,r,N}_{\infty}(A)\}$;
\item $GL^{\varepsilon,r,N}_{\infty}(A)/\sim:=\{[u]\mid u\in GL^{\varepsilon,r,N}_{\infty}(A)\}$ and $[u]+[v]=\begin{bmatrix}\begin{pmatrix}u&0\\0&v\end{pmatrix}\end{bmatrix}$;
\item $GL^{\varepsilon,r,N}_{\infty}(A)/\sim$ is an abelian group with identity $[1]$.
			\end{itemize}
If we want to take into account parameter changes, we usually write $[u]_{\varepsilon,r, N}$ instead of $[u]$.

\begin{definition}\cite{Banach-Chung}
Let $A$ be a unital filtered $L^{p}$ operator algebra. For $0<\varepsilon<\frac{1}{20}$, $r>0$ and $N\geq 1$,
\begin{itemize}
				\item	 if $A$ is unital, define $K^{\varepsilon,r,N}_{1}(A):= GL^{\varepsilon,r,N}_{\infty}(A)/\sim$;
          \item	if $A$ is non-unital, define $K^{\varepsilon,r,N}_{1}(A):= ker(\pi_{*}:K^{\varepsilon,r,N}_{1}(A^+)\rightarrow K^{\varepsilon,r,N}_{1}(\mathbb{C}))$.
\end{itemize}
\end{definition}

\begin{remark}\label{rmk 2.11}
If $e$ is an $(\varepsilon,r,N)$-idempotent in $A$, we can choose a function $\kappa_0$ that is holomorphic on a neighborhood of $Sp(e)$, and
 $$\kappa_{0}(z)=\begin{cases}
0, &z\in \bar{B}_{\sqrt{\varepsilon}}(0)\\
1, &z\in \bar{B}_{\sqrt{\varepsilon}}(1),
\end{cases} $$
then we apply holomorphic functional calculus to get an idempotent
$$\kappa_{0}(e)=\frac{1}{2\pi i}\int_{\gamma}\kappa_{0}(z)(z-e)^{-1}dz\in A,$$
and
$$\Vert\kappa_{0}(e)\Vert<\frac{N+1}{1-2\sqrt{\varepsilon}},$$
which implies that $\Vert\kappa_{0}(e)\Vert<2(N+1)$. Since each $(\varepsilon,r,N)$-invertible is invertible, we can define a function $\kappa_{1}$ such that $\kappa_{1}(u)=u$, thus $\Vert\kappa_{1}(u)\Vert\leq N$.
\end{remark}

\begin{definition}
For any filtered $L^p$ operator algebra $A$ and any positive numbers $r,r',\varepsilon,\varepsilon'$ and $N,N'\geq 1$ with $\varepsilon\leq\varepsilon'<\frac{1}{20}$, $r\leq r'$ and $N\leq N'$, we have natural group homomorphisms:
\begin{itemize}
\item $\iota_{0}: K_{0}^{\varepsilon,r,N}(A)\rightarrow K_{0}(A)$, $[e]_{\varepsilon,r,N}\mapsto [\kappa_{0}(e)]$;

\item $\iota_{1}: K_{1}^{\varepsilon,r,N}(A)\rightarrow K_{1}(A)$, $[u]_{\varepsilon,r,N}\mapsto [\kappa_{1}(u)]=[u]$;

\item $\iota_{*}= \iota_{0}\oplus\iota_{1}$;

\item $\iota^{\varepsilon',r',N'}_{0}: K^{\varepsilon,r,N}_{0}(A)\rightarrow K^{\varepsilon',r',N'}_{0}(A)$, $[e]_{\varepsilon,r,N}\mapsto [e]_{\varepsilon',r',N'}$;

\item $\iota^{\varepsilon',r',N'}_{1}: K^{\varepsilon,r,N}_{1}(A)\rightarrow K^{\varepsilon',r',N'}_{1}(A)$, $[u]_{\varepsilon,r,N}\mapsto [u]_{\varepsilon',r',N'}$;

\item $\iota^{\varepsilon',r',N'}_{*} =\iota^{\varepsilon',r',N'}_{0}\oplus\iota^{\varepsilon',r',N'}_{1}$.
\end{itemize}
\end{definition}
\begin{remark}
 We sometimes refer to these natural homomorphisms as relaxation of control maps. In addition, from the above definition, we know that the origin of variable parameters of quasi-idempotents or quasi-invertibles, thus we only mark the destination of the parameters to reduce to three superscripts.
\end{remark}

\begin{proposition}\cite{Banach-Chung}\label{prop 2.11}
There exists a polynomial $\rho\geq 1$ with positive coefficients such that for any filtered $L^{p}$ operator algebra $A$, any $\varepsilon\in (0,\frac{1}{20\rho(N)})$, any $r>0$ and any $N\geq 1$, the following holds:

Let $[x], [x']$ be in $K^{\varepsilon,r,N}_{*}(A)$ such that $\iota_{*}([x])=\iota_{*}([x'])$ in $K_{*}(A)$, there exist $r'\geq r$ and $N'\geq N$ such that 
$$\iota^{\rho(N)\varepsilon, r',N'}_{*}([x])=\iota^{\rho(N)\varepsilon, r',N'}_{*}([x']) \text{ in } K^{\rho(N)\varepsilon, r',N'}_{*}(A).$$
\end{proposition}

\begin{remark}\label{rmk 2.13}
From the proof of Proposition 3.21 in \cite{Banach-Chung}, we know that the choice of $N'$ depends on the norm of the homotopy path of the idempotents or invertibles, and we can choose 
$$\rho(N)=\begin{cases}
1+\frac{9}{20}(N+1)^{2}, \quad &*=0\\
1, \quad &*=1.
\end{cases} $$
\end{remark}

The item (ii) of the next proposition is a consequence of the preceding proposition. 
\begin{proposition}\cite{Banach-Chung}\label{prop 2.12}
Let $A$ be an $L^{p}$ operator algebra filtered by $(A_{r})_{r>0}$.

$\mathrm(i)$ For any $\varepsilon\in (0,\frac{1}{20})$ and any $[y]\in K_{*}(A)$, there exist $r>0$, $N\geq 1$ and $[x]\in K^{\varepsilon,r,N}_{*}(A)$ such that $\iota_{*}([x])=[y]$.

$\mathrm(ii)$ There exists a polynomial $\rho\geq 1$ with positive coefficients such that the following is satisfied: for $\varepsilon\in (0,\frac{1}{20\rho(N)})$, $r>0$ and $N\geq 1$, let $[x]$ be an element of $ K^{\varepsilon,r,N}_{*}(A)$ such that $\iota_{*}([x])=0$ in $K _{*}(A)$. Then there exist $r'\geq r$ and $N'\geq N$ such that
$$\iota^{\rho(N)\varepsilon, r', N'}_{*}([x])=0 \text{ in } K^{\rho(N)\varepsilon, r', N'}_{*}(A).$$
\end{proposition}
\begin{remark}\label{rmk 2.15}
From the proof of Proposition 3.20 in \cite{Banach-Chung}, we may put 
$$N=\begin{cases}
\Vert y\Vert+1, &*=0\\
\Vert y\Vert+\Vert y^{-1}\Vert+1, &*=1
\end{cases}$$ 
in the item (i) of the above proposition.
\end{remark}

\begin{definition}\cite{Banach-Chung}
 A control pair is a pair $(\lambda, h)$ such that
\begin{itemize}
\item $\lambda:[1,\infty)\rightarrow [1,\infty)$ is a non-decreasing function;

\item $h:(0,\frac{1}{20})\times [1,\infty)\rightarrow [1,\infty)$ is a function such that $h(\cdot,N)$ is non-increasing for fixed $N$.
\end{itemize}
We will write $\lambda_{N}$ for $\lambda(N)$, and $h_{\varepsilon, N}$ for $h(\varepsilon, N)$. Given two control pairs $(\lambda, h)$ and $(\lambda', h')$, we say that $(\lambda, h)\leq (\lambda', h')$ if $\lambda_{N}\leq \lambda'_{N}$  and $h_{\varepsilon, N}\leq h'_{\varepsilon, N}$ for all $\varepsilon\in(0, \frac{1}{20})$ and $N\geq 1$.
\end{definition}
Given a filtered $L^{p}$ operator algebra $A$, we write the families
$$\mathcal{K}_{i}(A)=(K^{\varepsilon,r,N}_{i}(A))_{0<\varepsilon<\frac{1}{20},r>0,N\geq 1},\text{ where } i\in\{0,1\}.$$

\begin{definition}\cite{Banach-Chung}
Let $A$ and $B$ be filtered $L^{p}$ operator algebras, and let $(\lambda,h)$ be a control pair. A $(\lambda,h)$-controlled morphism $\mathcal{F}:\mathcal{K}_{i}(A)\rightarrow \mathcal{K}_{j}(B)$, where $i,j\in\{0,1\}$, is a family 
$$\mathcal{F}=(F^{\varepsilon,r,N})_{0<\varepsilon<\frac{1}{20\lambda_{N}},r>0, N\geq 1}$$
of group homomorphisms
$$F^{\varepsilon,r,N}:K^{\varepsilon,r,N}_{i}(A)\rightarrow K_{j}^{\lambda_{N}\varepsilon,h_{\varepsilon,N}r,\lambda_{N}}(B)$$
such that whenever $0<\varepsilon\leq\varepsilon'<\frac{1}{20\lambda_{N'}}$, $h_{\varepsilon,N}r\leq h_{\varepsilon',N'}r'$ and $N\leq N'$, we have the following commutative diagram:
$$\begin{CD}
K^{\varepsilon,r,N}_{i}(A)@>{\iota_{i}}>>K^{\varepsilon',r',N'}_{i}(A)
\\@V{F^{\varepsilon,r,N}}VV@VV{F^{\varepsilon',r',N'}}V\\
K^{\lambda_{N}\varepsilon,h_{\varepsilon,N}r,\lambda_{N}}_{j}(B)@>{\iota_{j}}>>K^{\lambda_{N'}\varepsilon',h_{\varepsilon',N'}r',\lambda_{N'}}_{j}(B).
\end{CD}$$
We write $\iota_{i}$ for $\iota^{\varepsilon',r',N'}_{i}$ and $\iota_{j}$ for $\iota^{\lambda_{N'}\varepsilon',h_{\varepsilon',N'}r',\lambda_{N'}}_{j}$.
We say that $\mathcal{F}$ is a controlled morphism if it is a $(\lambda,h)$-controlled morphism for some control pair $(\lambda,h)$.
\end{definition}

\begin{definition}\cite{Banach-Chung}
Let $A$ and $B$ be filtered $L^{p}$ operator algebras. Let $\mathcal{F}:\mathcal{K}_{i}(A)\rightarrow \mathcal{K}_{j}(B)$ and $\mathcal{G}:\mathcal{K}_{i}(A)\rightarrow \mathcal{K}_{j}(B)$ be $(\lambda^{\mathcal{F}},h^{\mathcal{F}})$-controlled and $(\lambda^{\mathcal{G}},h^{\mathcal{G}})$-controlled morphisms respectively. Let $(\lambda,h)$ be a control pair. We write $\mathcal{F}\overset{(\lambda,h)}{\sim}\mathcal{G}$ if $(\lambda^{\mathcal{F}},h^{\mathcal{F}})\leq (\lambda,h), (\lambda^{\mathcal{G}},h^{\mathcal{G}})\leq (\lambda,h)$, and the following diagram commutes whenever $0<\varepsilon<\frac{1}{20\lambda_{N}},r>0$, and $N\geq 1$:
\begin{center}
\quad\xymatrix@R=0.5cm{
          &      K_j^{\lambda_{N}^{\mathcal{F}}\varepsilon,h^\mathcal{F}_{\varepsilon,N}r,\lambda_N^\mathcal{F}}(B)\ar[dr]^{\iota_j}     \\
\quad\quad\quad K_i^{\varepsilon,r,N}(A)\ar[ur]^{F^{\varepsilon,r,N}} \ar[dr]_{G^{\varepsilon,r,N}}&&K_j^{{\lambda_N}\varepsilon,h_{\varepsilon,N}r,\lambda_N}(B)               \\
              &      K_j^{\lambda_{N}^{\mathcal{G}}\varepsilon,h_{\varepsilon,N}^{\mathcal{G}}r,\lambda_N^{\mathcal{G}}}(B) \ar[ur]_{\iota_j}              .}
\end{center}
Observe that if $\mathcal{F}\overset{(\lambda,h)}{\sim}\mathcal{G}$ for some control pair $(\lambda,h)$, then $\mathcal{F}$ and $\mathcal{G}$ induce the same homomorphism in $K$-theory.
\end{definition}

\begin{definition}\cite{Banach-Chung}
Let $A$ and $B$ be filtered $L^{p}$ operator algebras. Let $(\lambda,h)$ be a control pair, and let $\mathcal{F}:\mathcal{K}_{i}(A)\rightarrow \mathcal{K}_{j}(B)$ be a $(\lambda^\mathcal{F},h^\mathcal{F})$-controlled morphism with $(\lambda^\mathcal{F}, h^\mathcal{F})\leq(\lambda,h).$
\begin{itemize}
\item We say that $\mathcal{F}$ is left (resp. right) $(\lambda,h)$-invertible if there exists a controlled morphism $\mathcal{G}:\mathcal{K}_{j}(B)\rightarrow \mathcal{K}_{i}(A)$ such that $\mathcal{G}\circ\mathcal{F}\overset{(\lambda,h)}{\sim}\mathcal{I}d_{\mathcal{K}_{i}(A)}$ (resp. $\mathcal{F}\circ\mathcal{G}\overset{(\lambda,h)}{\sim}\mathcal{I}d_{\mathcal{K}_{j}(B)}$). In this case, we call $\mathcal{G}$ a left (resp. right) $(\lambda,h)$-inverse for $\mathcal{F}$.
\item We say that $\mathcal{F}$ is $(\lambda,h)$-invertible or a $(\lambda,h)$-isomorphism if there exists a controlled morphism $\mathcal{G}:\mathcal{K}_{j}(B)\rightarrow \mathcal{K}_{i}(A)$ that is both a left $(\lambda,h)$-inverse and a right $(\lambda,h)$-inverse for $\mathcal{F}$. In this case, we call $\mathcal{G}$ a $(\lambda,h)$-inverse for $\mathcal{F}$.
\end{itemize}
We say that $\mathcal{F}$ is a controlled isomorphism if it is a $(\lambda,h)$-isomorphism for some control pair $(\lambda,h)$.
\end{definition}

\begin{definition}\cite{Banach-Chung}
Let $A$ and $B$ be filtered $L^{p}$ operator algebras. Let $(\lambda,h)$ be a control pair, and let $\mathcal{F}:\mathcal{K}_{i}(A)\rightarrow \mathcal{K}_{j}(B)$ be a $(\lambda^\mathcal{F}, h^\mathcal{F})$-controlled morphism with $(\lambda^\mathcal{F},h^\mathcal{F})\leq (\lambda,h)$.
\begin{itemize}
\item We say that $\mathcal{F}$ is $(\lambda,h)$-injective if for any $0<\varepsilon<\frac{1}{20\lambda_{N}}$, $r>0$, $N\geq 1$, and $[x]\in K^{\varepsilon,r,N}_{i}(A)$, if $F^{\varepsilon,r,N}([x])=0$ in $K^{\lambda^{F}_{N}\varepsilon,h^{\mathcal{F}}_{\varepsilon,N}r,\lambda^{F}_{N}}_{j}(B)$, then $\iota^{\lambda_{N}\varepsilon,h_{\varepsilon,N}r,\lambda_{N}}_{i}([x])=0$ in $K_{i}^{\lambda_{N}\varepsilon,h_{\varepsilon,N}r,\lambda_{N}}(A)$.

\item We say that $\mathcal{F}$ is $(\lambda,h)$-surjective if for any $0<\varepsilon<\frac{1}{20(\lambda_{F}\cdot\lambda)_{N}}$, $r>0$, $N\geq 1$, and $[y]\in K^{\varepsilon,r,N}_{j}(B)$, there exists $[x]\in K_{i}^{\lambda_{N}\varepsilon,h_{\varepsilon,N}r,\lambda_{N}}(A)$ such that 
$$F^{\lambda_{N}\varepsilon,h_{\varepsilon,N}r,\lambda_{N}}([x])=\iota^{(\lambda^{\mathcal{F}}\cdot\lambda)_{N}\varepsilon,(h^{\mathcal{F}}\cdot h)_{\varepsilon,N}r,(\lambda^{\mathcal{F}}\cdot\lambda)_{N}}_{j}([y])\text{ in }K_{j}^{(\lambda^{\mathcal{F}}\cdot\lambda)_{N}\varepsilon,(h^{\mathcal{F}}\cdot h)_{\varepsilon,N}r,(\lambda^{\mathcal{F}}\cdot\lambda)_{N}}(B).$$
\end{itemize}
\end{definition}

\begin{proposition}\cite{Banach-Chung}\label{prop 2.19}
Let $A$ be a unital filtered $L^{p}$ operator algebra. 

$\mathrm (i)$ If $e$ and $f$ are homotopic as $(\varepsilon,r,N)$-idempotents in $A$, then there exist $\alpha_N>0$, an integer $ k$ and an $\alpha_N$-Lipschitz homotopy of $(2\varepsilon,r,\frac{5}{2}N)$-idempotents between $diag(e,I_k,0_k)$ and diag$(f,I_k,0_k)$.

$\mathrm (ii)$ If $u$ and $v$ are homotopic as $(\varepsilon,r,N)$-invertibles in $A$, then there exist $\beta_N>0$, an integer $ k$ and a $\beta_N$-Lipschitz homotopy of $((4N^{2}+2)\varepsilon,2r,2(N+\varepsilon))$-invertibles between $diag(u,I_k)$ and diag$(v,I_k)$.
\end{proposition}

\begin{remark}
In fact, the proof of item (ii) is similar to that of item (i) [\cite{Banach-Chung}, Lemma 2.29]. 
\end{remark}

\begin{remark}
Let $A$ be an $L^{p}$ operator algebra, and let $\otimes$ denote the spatial $L^{p}$ operator tensor product. $M_{n}(A)$ can be regarded as $M_{n}(C)\otimes A$ when $M_{n}(\mathbb{C})$ is viewed as $B(\bigoplus\limits_{i=1}^{n}\ell^{p})$. Recall from Proposition 1.8 and Example 1.10 in \cite{Phillips}, we see that $\overline{M^{p}_{\infty}}=\mathscr{K}(\ell^{p})$ for $p\in (1,\infty)$ when $\overline{M^{p}_{\infty}}$ denotes $\overline{\bigcup\limits_{n\in\mathbb{N}} M_{n}(\mathbb{C})}^{\Vert\cdot\Vert_{\ell^{p}}}$. However, when $p=1$, there is a rank one operator on $\ell^{1}$ that is not in $\overline{M^{1}_{\infty}}$.

Now we collect some concepts of \cite{Phillips} concerning $L^{p}$ operator tensor products. For $p\in[1,\infty)$ and for measure spaces $(X,\mu)$ and $(Y,\nu)$, there is an $L^{p}$ tensor product such that we have a canonical isometric isomorphism $L^{p}(X,\mu)\otimes L^{p}(Y,\nu)\cong L^{p}(X\times Y,\mu\times\nu)$ via $(x,y)\mapsto\xi(x)\eta(y)$ for any $\xi\in L^{p}(X,\mu)$, $\eta\in L^{p}(Y,\nu)$, this tensor product has the following properties:
\begin{itemize}
\item Under the previous isomorphism, the linear span of all $\xi\otimes\eta$ is dense in $L^{p}(X\times Y,\mu\times\nu)$;
\item $\Vert\xi\otimes\eta\Vert_{p}=\Vert\xi\Vert_{p}\Vert\eta\Vert_{p}$ for all $\xi\in L^{p}(X,\mu)$ and $\eta\in L^{p}(Y,\nu)$;
\item The tensor product is commutative and associative;
\item If $a\in B\big(L^{p}(X_{1},\mu_{1}),L^{p}(X_{2},\mu_{2})\big)$ and $b\in B\big(L^{p}(Y_{1},\nu_{1}),L^{p}(Y_{2},\nu_{2})\big)$, then there exists a unique 
$$c\in B\big(L^{p}(X_{1}\times Y_{1},\mu_{1}\times\nu_{1}), L^{p}(X_{2}\times Y_{2}, \mu_{2}\times\nu_{2})\big)$$
such that $c(\xi\otimes\eta)=a(\xi)\otimes b(\eta)$ for all $\xi\in L^{p}(X_{1},\mu_{1})$ and $\eta\in L^{p}(Y_{1},\nu_{1})$. We will denote this operator by $a\otimes b$, thus $\Vert a\otimes b\Vert=\Vert a\Vert\Vert b\Vert$;
\item The tensor product of operators is associative, bilinear, and satisfies $(a_{1}\otimes b_{1})(a_{2}\otimes b_{2})=a_{1}a_{2}\otimes b_{1}b_{2}$.
\end{itemize}

If $A\subset B\big(L^{p}(X,\mu)\big)$ and $B\subset B\big(L^{p}(Y,\nu)\big)$ are norm-closed subalgebras, we can define $A\otimes B\subset B(L^{p}(X\times Y,\mu\times\nu))$ to be the closed linear span of all elements of the form $a\otimes b$ with $a\in A$ and $b\in B$.
\end{remark}

\begin{proposition}\cite{Banach-Chung}\label{prop2}
If $A$ is a filtered $L^{p}$ operator algebra for some $p\in(1,\infty)$, then the homomorphism
$$A\rightarrow \mathscr{K}(\ell^{p})\otimes A,\quad a\mapsto \begin{pmatrix}a&&\\&0&\\&&\ddots\end{pmatrix}$$
induces a group isomorphism (the Morita equivalence)
$$K^{\varepsilon,r,N}_{*}(A)\rightarrow K^{\varepsilon,r,N}_{*}(\mathscr{K}(\ell^{p})\otimes A).$$
\end{proposition}

For $p=1$, we denote $\mathscr{K}(\ell^{1})$ by $\overline{\bigcup\limits_{n\in\mathbb{N}}M_{n}(\mathbb{C})}^{\Vert\cdot\Vert_{\ell^{1}}}$, then we still have the Morita equivalance.
\begin{proposition}\cite{Banach-Chung}\label{prop 2.17}
If $A$ is a filtered $L^{1}$ operator algebra,  then we have a group isomorphism 
$$K^{\varepsilon,r,N}_{*}(\mathscr{K}(\ell^{1})\otimes A)\cong K^{\varepsilon,r,N}_{*}(A).$$
\end{proposition}
\begin{remark}
For any $r>0$, the $L^{p}$ operator tensor product $\mathscr{K}(\ell^{p})\otimes A$ has a filtration $(\mathscr{K}(\ell^{p})\otimes A_{r})_{r>0}$.
\end{remark}
If $\mathcal{A}=(A_{i})_{i\in\mathbb{N}}$ is any family of filtered $L^{p}$ operator algebras. For any $r>0$, we set 
$$\mathcal{A}^{\infty}_{c,r}=\prod\limits_{i\in\mathbb{N}}\mathscr{K}(\ell^{p})\otimes A_{i,r},$$
and we define the $L^{p}$ operator algebra $\mathcal{A}^{\infty}_{c}$ as the closure of $\bigcup\limits_{r>0}\mathcal{A}^{\infty}_{c,r}$ in $\prod\limits_{i\in\mathbb{N}}\mathscr{K}(\ell^{p})\otimes A_{i}$.
\begin{lemma}\label{lemma 2}
Let $\mathcal{A}=(A_{i})_{i\in\mathbb{N}}$ be a family of filtered $L^{p}$ operator algebras. There exist a control pair $(\lambda,h)$ independent of the family $\mathcal{A}$ and a $(\lambda,h)$-isomorphism 
$$\mathcal{F}=(F^{\varepsilon,r,N})_{0<\varepsilon<\frac{1}{20},r>0,N\geq 1}: \mathcal{K}_{*}(\mathcal{A}^{\infty}_{c})\rightarrow\prod\limits_{i\in\mathbb{N}}\mathcal{K}_{*}(A_{i}),$$
where 
$$F^{\varepsilon,r,N}:K^{\varepsilon,r,N}_{*}(\mathcal{A}^{\infty}_{c})\rightarrow\prod\limits_{i\in\mathbb{N}}K^{\varepsilon,r,N}_{*}(A_{i}) $$
is induced on the $j$-th factor by the projection $\prod\limits_{i\in\mathbb{N}}\mathscr{K}(\ell^{p})\otimes A_{i}\rightarrow \mathscr{K}(\ell^{p})\otimes A_{j}$  and up to the Morita equivalence restricted to $\mathcal{A}^{\infty}_{c}$.
\end{lemma}

\begin{remark}
If $A_{i}$ is unital, then the above Lemma \ref{lemma 2} is a consequence of Proposition \ref{prop 2.19}. In this case, we let $\lambda_{N}=\frac{5}{2}N$, $h(\cdot,N)=2$. If $A_{i}$ is not unital for some $i$, the proof is similar to that of Lemma 2.14 in \cite{PAP-Oyono}.
\end{remark}
\section{Quantitative $L^p$ assembly maps}
In this section, we will introduce $L^{p}$ localization algebras, $L^{p}$ Roe algebras and reduced $L^{p}$ crossed products to define quantitative $L^{p}$ assembly maps, and establish the connection between the $L^p$ Baum-Connes conjecture and the quantitative $L^p$ Baum-Connes conjecture.

\subsection{$L^p$ Roe algebras and $L^p$ localization algebras}
In this section, we consider the case of finitely generated groups. Let $\Gamma$ be a finitely generated group with a length function $\ell:\Gamma\rightarrow\mathbb{R}^{+}$ such that 
\begin{itemize}
\item $\ell(\gamma)=0$ if and only if $\gamma=e$, where $e$ is the identity element of $\Gamma$; 

\item $\ell(\gamma\gamma')\leq\ell(\gamma)+\ell(\gamma')$ for all $\gamma,\gamma'\in\Gamma$;

\item $\ell(\gamma)=\ell(\gamma^{-1})$ for all $\gamma\in\Gamma$.
\end{itemize}
We assume that $\ell$ is the word length
$$\ell(\gamma)=\mathrm{inf}\{d \mid \gamma=\gamma_{1}\cdots\gamma_{d} \text{ with }\gamma_{1},\cdots,\gamma_{d}\in S\},$$
where $S$ is a finitely generated symmetric set. Let the ball of radius $r\in (0,\infty)$ around the identity of $\Gamma$ be 
$$B(e,r)=\{\gamma\in\Gamma\mid \ell(\gamma)\leq r\}.$$

\begin{definition}\cite{Higher-Yu}
Let $\Gamma$ be a finitely generated group and let $d\geq 0$. The spherical Rips complex of $\Gamma$ at scale $d$, denoted by $S_{d}(\Gamma)$, consists as a set of all formal sums
$$x=\sum\limits_{\gamma\in\Gamma}t_{\gamma}\gamma$$
such that each $t_{\gamma}\in [0,1]$ with $\sum\limits_{\gamma\in\Gamma}t_{\gamma}=1$ and such that the support of $x$ defined by $$\supp(x):=\{\gamma\in\Gamma\mid t_{\gamma}\neq 0\}$$ has diameter at most $d$.
\end{definition}

\begin{definition}\cite{Higher-Yu}
Let $\Gamma$ be a finitely generated group, and let $S_d(\Gamma)$ be the associated spherical Rips complex at scale $d$. A semi-simplicial path $\delta$ between points $x$ and $y$ in $S_d(\Gamma)$ consists of a sequence of the form
$$x=x_{0},y_{0},x_{1},y_{1},x_{2},y_{2},\cdots,x_{n},y_{n}=y,$$
where each of $x_{1},\cdots,x_{n}$ and each of $y_{0},\cdots,y_{n-1}$ are in $\Gamma$. The length of such a path is 
$$l(\delta):=\sum\limits_{i=0}^{n}d_{S_{d}}(x_{i},y_{i})+\sum\limits_{i=0}^{n-1}d_{\Gamma}(y_{i},x_{i+1}).$$
We define the semi-spherical distance on $S_d(\Gamma)$ by
$$d_{P_d}(x,y):=\mathrm{inf}\{l(\gamma) \mid \gamma \text{ is a semi-simplicial path between } x \text{ and }y\}$$
(note that a semi-simplicial path between two points always exists).

The Rips complex of $\Gamma$ is defined to be the space $P_d(\Gamma)$ equipped with the
metric $d_{P_d}$ above. 
\end{definition}
\begin{remark}
$P_{d}(\Gamma)$ is a locally finite simplicial complex and is locally compact when endowed with the simplicial topology, and it is endowed with a proper and cocompact action of $\Gamma$ by left translation.
\end{remark}
\begin{definition}
For $d\geq 0$, we define
$$Q_{d}:=\left\{\displaystyle\sum\limits_{\gamma\in\Gamma}t_{\gamma}\gamma\in P_{d}(\Gamma) \mid t_{\gamma}\in\mathbb{Q} \text{ for all } \gamma\in\Gamma \right\}.$$
Then $Q_{d}$ is a $\Gamma$-invariant, countable, dense subset of $P_{d}(\Gamma)$.
\end{definition}

\begin{definition}
Let $\Gamma$ be a discrete group, and let $A$ be an $L^p$ operator algebra. We say that $A$ is a $\Gamma$-$L^p$ operator algebra if $\alpha:\Gamma\rightarrow Aut(A)$ is an action by isometric automorphisms.
\end{definition}

\begin{definition}\cite{Phillips}
Let $(\Gamma, A, \alpha)$ be a $\Gamma$-$L^{p}$ operator algebra, and let $(X, \mathcal{B}, \mu)$ be a measure space. Then a covariant representation of $(\Gamma, A, \alpha)$ on $L^{p}(X, \mu)$ is a pair $(v, \pi)$ consisting of a representation $\gamma\mapsto v_{\gamma}$ from $\Gamma$ to the invertible operators on $L^p(X, \mu)$ such that $\gamma\mapsto v_{\gamma}\xi$ is continuous for all $\xi\in L^{p}(X, \mu)$, and a representation  $\pi: A \rightarrow B(L^p(X, \mu))$ such that the following covariance condition is satisfied: $\pi(\alpha_{\gamma}(a)) = v_{\gamma}\pi(a)v^{-1}_{\gamma}$
for all $\gamma\in\Gamma$ and $a\in A$. 
\end{definition}
We say that a covariant representation is isometric if $\pi$ is isometric.
\begin{definition}\label{def 3.4}
Let $A$ be a $\Gamma$-$L^{p}$ operator algebra, and let $E$ be a covariant represented $L^{p}$ space of $A$. An $L^{p}$-module is defined to be an $L^{p}$ space
$$L_{d}=\ell^{p}(Q_{d})\otimes E\otimes\ell^{p}\otimes\ell^{p}(\Gamma)\cong\ell^{p}(Q_{d}, E\otimes\ell^{p}\otimes\ell^{p}(\Gamma))$$
equipped with an isometric $\Gamma$-action given by 
$$u_{\gamma}\cdot(\delta_{x}\otimes e\otimes\eta\otimes\delta_{\gamma'})=\delta_{x\gamma^{-1}}\otimes\gamma e\otimes\eta\otimes\delta_{\gamma\gamma'},$$
for $x\in Q_{d}, e\in E, \eta\in\ell^{p}$, and $\gamma,\gamma'\in\Gamma$.
\end{definition}

\begin{remark}
For each $d\geq d_{0}\geq 0$, the canonical inclusion $i_{d_{0},d}: P_{d_{0}}(\Gamma)\hookrightarrow P_{d}(\Gamma)$ is a homeomorphism on its image and a coarse equivalence, and $Q_{d_{0}}\subset Q_{d}$.  Hence, we have an equivariant isometric inclusion $L_{d_{0}}\subset L_{d}$.
\end{remark}

\begin{remark}
Let $\mathscr{K}_{\Gamma}$ be the algebra of compact operators on $\ell^{p}\otimes\ell^{p}(\Gamma)\cong\ell^{p}(\mathbb{N}\times\Gamma)$ equipped with the $\Gamma$-action induced by the tensor product of the trivial action on $\ell^{p}$ and the left regular representation on $\ell^{p}(\Gamma)$. Also, we equip the algebra $A\otimes\mathscr{K}_{\Gamma}$ with the diagonal action of $\Gamma$. We say that the representation of $A\otimes\mathscr{K}_{\Gamma}$ on $E\otimes\ell^{p}\otimes\ell^{p}(\Gamma)$ is faithful and covariant if this representation is obtained by tensoring the natural action on $E$, trivial on $\ell^{p}$, and regular on $\ell^{p}(\Gamma)$. 
\end{remark}
Next, we will define equivariant $L^{p}$ Roe algebras and equivariant $L^{p}$ localization algebras.
\begin{definition}
Let $L_{d}$ be the $L^{p}$-module as in Definition \ref{def 3.4}, and let $T$ be a bounded linear operator on $L_{d}$, which we regard as a $(Q_{d}\times Q_{d})$-indexed matrix $T=(T_{y,z})$ with
$$T_{y,z}\in B(E\otimes\ell^{p}\otimes\ell^{p}(\Gamma))$$
for all $y,z\in Q_{d}$.
\begin{itemize}
\item $T$ is $\Gamma$-invariant if $u_{\gamma}Tu^{-1}_{\gamma}=T$ for all $\gamma\in\Gamma$, i.e. $T_{y,z}=\gamma\cdot T_{y\gamma,z\gamma}$ for all $\gamma\in\Gamma$.

\item The propagation of $T$ is defined to be 
$$prop(T):=\sup\{d_{P_{d}(\Gamma)}(y,z): T_{y,z}\neq 0\}.$$
\item $T$ is $E$-locally compact if $T_{y,z}\in A\otimes\mathcal{K}_{\Gamma}$ for all $y,z\in Q_{d}$, and if for each compact subset $G\subset P_{d}(\Gamma)$, the set
$$\{(y,z)\in (G\times G)\cap (Q_{d}\times Q_{d}): T_{y,z}\neq 0\}$$
is finite.
\end{itemize} 
\end{definition}

\begin{definition}
Let $L_{d}$ be the $L^{p}$-module, and let $\mathbb{C}[L_{d},A]^{\Gamma}$ denote the algebra of all $\Gamma$-invariant, $E$-locally compact operators on $L_{d}$ with finite propagation. The equivariant $L^{p}$ Roe algebra with coefficients in $A$, denoted $B^{p}(P_{d}(\Gamma),A)^{\Gamma}$, is defined to be closure of $\mathbb{C}[L_{d},A]^{\Gamma}$ in the operator norm on $B(L_{d})$.
\end{definition}

\begin{definition}
Let $L_{d}$ be the $L^{p}$-module, and let $\mathbb{C}_{L}[L_{d},A]^{\Gamma}$ denote the algebra of all bounded, uniformly continuous functions $f:[0,\infty)\rightarrow\mathbb{C}[L_{d},A]^{\Gamma}$ such that 
 $$prop(f(t))\rightarrow 0\text{ as }t\rightarrow\infty.$$
The equivariant $L^{p}$ localization algebra with coefficients in $A$, denoted by $B^{p}_{L}(P_{d}(\Gamma),A)^{\Gamma}$, is the completion of $\mathbb{C}_{L}[L_{d},A]^{\Gamma}$ with respect to the norm
$$\Vert f\Vert:=\sup\limits_{t\in [0,\infty)}\Vert f(t)\Vert_{B(L_{d})}.$$
\end{definition}

\subsection{The quantitative $L^p$ assembly maps}
For $p\in [1,\infty)$, to define a quantitative $L^p$ assembly map, we replace the equivariant $KK$-theory by the equivariant $K$-theory of $L^{p}$ localization algebras on the left-hand side of the map and replace the reduced $C^{*}$ crossed product by the reduced $L^{p}$ crossed product on the right-hand side of the map. In the setting of $L^{p}$ operator algebras, we need to study reduced $L^{p}$ crossed products and $L^{p}$ Baum-Connes assembly maps.

\begin{definition}
Let $A$ be a $\Gamma$-$L^{p}$ operator algebra, and let $E$ be an $L^{p}$ representation space of $A$. The reduced $L^{p}$ crossed product $A\rtimes_{\alpha,\lambda}\Gamma$ is the completion of $C_{c}(\Gamma,A,\alpha)$ in the operator norm on $B(E\otimes\ell^{p}(\Gamma))$.
\end{definition}

\begin{remark}
If $A$ is a matrix algebra $M_{n}(\mathbb{C})$ or a commutative algebra $C(X)$ for some compact space $X$, then the above definition is identical with Phillips's reduced $L^p$ crossed products [\cite{Phillips}, Definition 3.3] since it is independent of the representation of $A$ [\cite{WZ23}, Lemma 2.6].

\end{remark}

\begin{remark}
In the following, we will write $A\rtimes\Gamma$ for $A\rtimes_{\alpha,\lambda}\Gamma$. Note that the identification between $A\rtimes\Gamma$ and $B^{p}(P_{d}(\Gamma),A)^{\Gamma}$ is derived from the Morita equivalence between $C_{c}(\Gamma,A,\alpha)$ and $\mathbb{C}[L_{d},A]^{\Gamma}$. In addition, for $r>0$, the reduced $L^{p}$ crossed product $A\rtimes\Gamma$  has a filtration 
$$(A\rtimes\Gamma)_{r}:=\{f\in C_{c}(\Gamma,A) \text{ with } \mathrm{supp}(f)\in B(e,r)\}.$$
\end{remark}

\begin{definition}
Let $A$ be an $L^{p}$ operator algebra. For $N\geq 1$,
\begin{itemize}
\item an element $z\in A$ is called an $N$-idempotent if $z^{2}=z$ and $\Vert z\Vert\leq N$;
\item if $A$ is unital, an element $w\in A$ is called an $N$-invertible if $w$ is invertible and $\max\{\Vert w\Vert,\Vert w^{-1}\Vert\}\leq N$.
\end{itemize}
\end{definition}
Then we will define a variant of $K$-theory of $L^{p}$ operator algebras, which is labeled by the norm of the element and the norm of the homotopy path. 

Given an $L^{p}$ operator algebra $A$, for $N\geq 1$:
\begin{itemize}
\item we set $Idem^{N}(A):=\{z\in A \mid z$ is an $N$-idempotent\};
\item we let $Idem^{N}_{m}(A)=Idem^{N}(M_{m}(A))$ for each $m\in\mathbb{N}$;
\item we have inclusions $Idem^{N}_{m}(A)\hookrightarrow Idem^{N}_{m+1}(A)$, $ z\mapsto\begin{pmatrix}z&0\\0&0\end{pmatrix}$; 
\item we put $Idem^{N}_{\infty}(A):=\bigcup\limits_{m\in\mathbb{N}}Idem^{N}_{m}(A)$;
\item we define the equivalence relation $\sim$ on $Idem^{N}_{\infty}(A)$ as follows: $z\sim z'$ if $z$ and $z'$ are homotopic in $Idem^{4N}_{\infty}(A)$;
\item we denote by $[z]$ the equivalence class of $z\in Idem^{N}_{\infty}(A)$;
\item we equip $Idem^{N}_{\infty}(A)/\sim$ with the addition given by $[z]+[z']=[diag(z,z')]$;
\item $Idem^{N}_{\infty}(A)/\sim$ is an abelian semigroup with identity $[0]$.
\end{itemize}
If we wish to keep track of changes in the norm, we write $[z]_{N}$ instead of $[z]$.
\begin{definition}\label{def 3.13}
Let $A$ be an $L^{p}$ operator algebra. For $N\geq 1$, 
\begin{itemize}
\item if $A$ is unital, define $K^{N}_{0}(A)$ to be the Grothendieck group of $Idem^{N}_{\infty}(A)/\sim$;
\item if $A$ is non-unital, define $$K^{N}_{0}(A): =ker\big(\pi_{*}:K^{N}_{0}(A^{+})\rightarrow \mathbb{Z}\big).$$
\end{itemize}
\end{definition}

If $[z]-[z']\in K^{N}_{0}(A)$, where $z,z'\in M_{k}(\widetilde{A})$, then $[z]-[z']=[z'']-[I_{k}]$ in $K^{N}_{0}(A)$ for some $z''\in M_{2k}(\widetilde{A})$. Hence, each element of $K^{N}_{0}(A)$ can be written by $[z]-[I_{k}]$ with $\pi(z)=diag(I_{k},0)$.

Given a unital $L^{p}$ operator algebra $A$, for $N\geq 1$, 
\begin{itemize}
\item we set $GL^{N}(A):=\{w\in A\mid w$ is an $N$-invertible\};
\item we let $GL^{N}_{m}(A)=GL^{N}(M_{m}(A))$ for each $m\in\mathbb{N}$; 
\item we have inclusions $GL^{N}_{m}(A)\hookrightarrow GL^{N}_{m+1}(A)$, $w\mapsto\begin{pmatrix}w&0\\0&1\end{pmatrix}$; 
\item we put $GL^{N}_{\infty}(A):=\bigcup\limits_{m\in\mathbb{N}}GL^{N}_{m}(A)$;
\item we define the equivalence relation $\sim$ on $GL^{N}_{\infty}(A)$ as follows: $w\sim w'$ if $w$ and $w'$ are homotopic in $GL^{4N}_{\infty}(A)$;
\item we denote by $[w]$ the equivalence class of $w\in GL^{N}_{\infty}(A)$;
\item we equip $GL^{N}_{\infty}(A)/\sim$ with the addition defined by $[w]+[w']=[diag(w,w')]$;
\item $GL^{N}_{\infty}(A)/\sim$ is an abelian group with identity $[1]$.
\end{itemize}
If we wish to keep track of changes in the norm, we write $[w]_{N}$ instead of $[w]$.
\begin{definition}\label{def 3.15}
Let $A$ be an $L^{p}$ operator algebra. For $N\geq 1$,
\begin{itemize}
\item if $A$ is unital, define $K^{N}_{1}(A):= GL^{N}_{\infty}(A)/\sim$;
\item if $A$ is non-unital, define $K^{N}_{1}(A):= K^{N}_{1}(A^{+})$.
\end{itemize}
\end{definition}
In the odd case, each element of $K^{N}_{1}(A)$ can be written as $[w]$ with the form $\pi(w)=I_{k}$. 
Observe that $K^{N}_{*}(A)\subset K^{N'}_{*}(A)$ if $ N\leq N'$ and $K_{*}(A)=\lim\limits_{N\rightarrow\infty}K^{N}_{*}(A)$.

The evaluation-at-zero homomorphism 
$$ev_{0}: B^{p}_{L}\big(P_{d}(\Gamma),A\big)^{\Gamma}\rightarrow B^{p}\big(P_{d}(\Gamma),A\big)^{\Gamma},$$
induces a homomorphism on $K$-theory
$$ev_{*}: K_{*}\Big(B^{p}_{L}\big(P_{d}(\Gamma),A\big)^{\Gamma}\Big)\rightarrow  K_{*}\Big(B^{p}\big(P_{d}(\Gamma),A\big)^{\Gamma}\Big).$$

\begin{definition}\cite{DC-Chung}
Let $A$ be a  $\Gamma$-$L^p$ operator algebra. We define an $L^{p}$ assembly map
$$\mu^{d}_{A,*}:K_{*}\Big(B^{p}_{L}\big(P_{d}(\Gamma),A\big)^{\Gamma}\Big)\xrightarrow{ev_{*}}  K_{*}\Big(B^{p}\big(P_{d}(\Gamma),A\big)^{\Gamma}\Big)\cong K_{*}(A\rtimes\Gamma),$$
which gives rise to a homomorphism 
$$\mu_{A,*}:\lim\limits_{d>0} K_{*}\Big(B^{p}_{L}\big(P_{d}(\Gamma),A\big)^{\Gamma}\Big)\rightarrow K_{*}(A\rtimes\Gamma)$$
called the $L^p$ Baum-Connes assembly map. Moreover, the $L^{p}$ Baum-Connes conjecture for $\Gamma$ predicts that the $L^p$ Baum-Connes assembly map $\mu_{A,*}$ is an isomorphism.
\end{definition}

Subsequently, we will give a definition of a quantitative $L^{p}$ assembly map. Let us do some preparation. Consider the even case, the odd case being similar. Let $[z]$ be in $K^{N}_{0}\Big(B^{p}_{L}\big(P_{d}(\Gamma),A\big)^{\Gamma}\Big)$ with $z\in Idem^{N}_{m}\Big(B^{p}_{L}\big(P_{d}(\Gamma),A\big)^{\Gamma}\Big)$ for some $m$. Then for any $0<\varepsilon<\frac{1}{20}$, there exist $r'>0$, $\widetilde{z}\in Idem_{m}(\widetilde{\mathbb{C}_{L}[L_{d},A]^{\Gamma}_{r'}})$ such that $\Vert z-\widetilde{z}\Vert<\frac{\varepsilon}{6N(N+1)^{2}}$, then $\widetilde{z}$ is an $(\varepsilon,r',2N)$-idempotent in $M_{m}(\widetilde{\mathbb{C}_{L}[L_{d},A]^{\Gamma}})$ and $\iota_{0}([\widetilde{z}]_{\varepsilon,r',2N})=[z]$ [\cite{Banach-Chung}, Proposition 3.20]. Observe that the propagation of $\widetilde{z}$ tends to zero when $t$ goes to infinity. Hence, for $r>0$, we can choose $t\in [0,\infty)$ such that the $prop(\widetilde{z_{t}})\leq r$. Since $\Vert z_{t}-\widetilde{z_{t}}\Vert\leq \Vert z-\widetilde{z}\Vert<\frac{\varepsilon}{6N(N+1)^{2}}$, we get that $\widetilde{z_{t}}$ is an $(\varepsilon,r,2N)$-idempotent in $M_{m}(\widetilde{\mathbb{C}[L_{d},A]^{\Gamma}})$ and $\iota_{0}([\widetilde{z_{t}}]_{\varepsilon,r,2N})=[z_{t}]$ by applying Proposition 3.20 in \cite{Banach-Chung}.

\begin{definition}
Let $A$ be a $\Gamma$-$L^p$ operator algebra. For $0<\varepsilon<\frac{1}{20}$, $r>0$, $N\geq1$ and $d>0$, we define a quantitative $L^p$ assembly map

$$\mu^{\varepsilon,r,N,d}_{A,*}: K^{N}_{*}\Big(B^{p}_{L}\big(P_{d}(\Gamma),A\big)^{\Gamma}\Big)\rightarrow K^{\varepsilon,r,9N}_{*}\Big(B^{p}\big(P_{d}(\Gamma),A\big)^{\Gamma}\Big)\cong K^{\varepsilon,r,9N}_{*}(A\rtimes\Gamma)$$
$$\quad [z] \mapsto [\widetilde{z_{t}}]_{\varepsilon,r,9N}\qquad\qquad$$
for some $t\in [0,\infty)$ satisfying 
$$\iota_{*}([\widetilde{z_{t}}]_{\varepsilon,r,9N})=[z_{t}]\text{ in }K_{*}(A\rtimes\Gamma).$$
\end{definition}
\begin{remark}
Put $B=B^{p}_{L}\big(P_{d}(\Gamma),A\big)^{\Gamma}$.
In the even case, If $[z]=[z']\in K^{N}_{0}(B)$, then $[z]+[g]=[z']+[g]$ in $Idem^{N}_{\infty}(\widetilde{B})/\sim$ for some $g$ in $Idem^{N}_{k}(\widetilde{B})$, thus $diag(z,g)$ and $diag(z',g)$ are homotopic in $Idem^{4N}_{\infty}(\widetilde{B})$. Let $(Z^{s})_{s\in [0,1]}$ be a homotopy of $4N$-idempotents between $diag(z,g)$ and $diag(z',g)$, and let $0=s_{0}<s_{1}<\cdots<s_{k}=1$ be such that 
$$\Vert Z^{s_{i}}-Z^{s_{i-1}}\Vert<\frac{\varepsilon}{6(10N+1)}, \text{ for } i=1, \cdots, k.$$
For each $i$, there exist $r_{i}>0$, $\widetilde{Z^{s_{i}}}\in M_{m}(\widetilde{B_{r_{i}}})$ such that
 $\Vert Z^{s_{i}}-\widetilde{Z^{s_{i}}}\Vert<\frac{\varepsilon}{30N(5N+1)^{2}}$.
Then $\widetilde{Z^{s_{i}}}$ is an $(\varepsilon,r_{i},5N)$-idempotent in $ M_{m}(\widetilde{B})$ and $\iota_{0}([\widetilde{Z^{s_{i}}}])=[Z^{s_{i}}]$ in $K_{0}(B)$ [\cite{Banach-Chung}, Proposition 3.20]. For $r>0$, by the definition of the localization algebra, we can choose an appropriate $t_{i}$ in $[0,\infty)$ such that $Z^{s_{i}}_{t_{i}}$ is in $M_{m}(\widetilde{A\rtimes\Gamma})$ and the propagation of $Z^{s_{i}}_{t_{i}}$ is no more than $r$. Let $t=\max\limits_{0\leq i\leq k}t_{i}$, and define $\widetilde{Z^{l}_{t}}=\frac{l-s_{i-1}}{s_{i}-s_{i-1}}\widetilde{Z^{s_{i}}_{t}}+\frac{s_{i}-l}{s_{i}-s_{i-1}}\widetilde{Z^{s_{i-1}}_{t}}$ for $l\in [s_{i-1},s_{i}]$. Then $\widetilde{Z^{l}_{t}}$ is a homotopy of $(\varepsilon,r,5N)$-idempotent in  $M_{m}(\widetilde{A\rtimes\Gamma})$ between $\widetilde{Z^{0}_{t}}$ and $\widetilde{Z^{1}_{t}}$. The odd case is similar: we can also construct a homotopy of $(\varepsilon,r,9N)$-invertible in  $M_{m}(\widetilde{A\rtimes\Gamma})$. Note that $\max\{5N,9N\}=9N$. Hence, for any $[z]\in K^{N}_{*}(B)$, there exists a unique element $[\widetilde{z_{t}}]_{\varepsilon,r,9N}\in K^{\varepsilon,r,9N}_{*}(A\rtimes\Gamma)$ such that $\iota_{*}([\widetilde{z_{t}}]_{\varepsilon,r,9N})=[z_{t}]$ for some $t\in [0,\infty)$ in $K_{*}(A\rtimes\Gamma)$. Therefore, the quantitative $L^{p}$ assembly map $\mu^{\varepsilon,r,N,d}_{A,*}$ is well-defined. 
\end{remark}
Moreover, the quantitative $L^{p}$ assembly maps are compatible with the usual ones, namely, if $[z]$ is an element of $K^{N}_{*}\Big(B^{p}_{L}\big(P_{d}(\Gamma),A\big)^{\Gamma}\Big)$, then
\begin{equation}\label{eq 1}
\mu^{d}_{A,*}([z])=\iota_{*}\circ\mu^{\varepsilon,r,N,d}_{A,*}([z]_{N})\text{ in } K_{*}(A\rtimes\Gamma).
\end{equation}
 For any positive numbers $d, d'$ such that $d\leq d'$, we denote by 
$$i^{N}_{d,d',*}: K^{N}_{*}\Big(B^{p}_{L}\big(P_{d}(\Gamma),A\big)^{\Gamma}\Big)\rightarrow K^{N}_{*}\Big(B^{p}_{L}\big(P_{d'}(\Gamma),A\big)^{\Gamma}\Big).$$
the homomorphism induced by the canonical inclusion $i_{d,d'}: P_{d}(\Gamma)\hookrightarrow P_{d'}(\Gamma)$, then  
$$\mu^{\varepsilon,r,N,d}_{A,*}=\mu^{\varepsilon,r,N,d'}_{A,*}\circ i^{N}_{d,d',*},$$
which implies that $\mu^{d}_{A,*}=\mu^{d'}_{A,*}\circ i_{d,d',*}$.
Moreover, for $0<\varepsilon\leq\varepsilon'<\frac{1}{20}$, $0<r\leq r'$ and $1\leq N\leq N'$, we have
\begin{equation}\label{eq 2}
\iota^{\varepsilon',r',9N'}_{*}\circ\mu^{\varepsilon,r,N,d}_{A,*}=\mu^{\varepsilon',r',N',d}_{A,*}.
\end{equation}
For $N\geq 1$, the evaluation-at-zero homomorphism 
$$ev_{0}: B^{p}_{L}\big(P_{d}(\Gamma),A\big)^{\Gamma}\rightarrow B^{p}\big(P_{d}(\Gamma),A\big)^{\Gamma},$$
induces a homomorphism on a variant of $K$-theory
$$ev^{N}_{*}: K^{N}_{*}\Big(B^{p}_{L}\big(P_{d}(\Gamma),A\big)^{\Gamma}\Big)\rightarrow  K^{N}_{*}\Big(B^{p}\big(P_{d}(\Gamma),A\big)^{\Gamma}\Big).$$

\begin{definition}
Let $A$ be a $\Gamma$-$L^{p}$ operator algebra. For $N\geq 1$, we define an $N$-$L^{p}$ assembly map
$$\mu^{N,d}_{A,*}: K^{N}_{*}\Big(B^{p}_{L}(P_{d}\big(\Gamma),A\big)^{\Gamma}\Big)\xrightarrow {ev^{N}_{*}} K^{N}_{*}\Big(B^{p}\big(P_{d}(\Gamma),A\big)^{\Gamma}\Big)\cong K^{N}_{*}(A\rtimes\Gamma),$$
which gives rise to a homomorphism 
$$\mu^{N}_{A,*}: \lim\limits_{d>0}K^{N}_{*}\Big(B^{p}_{L}(P_{d}\big(\Gamma),A\big)^{\Gamma}\Big)\rightarrow K^{N}_{*}(A\rtimes\Gamma)$$
 called the $N$-$L^{p}$ Baum-Connes assembly map.
\end{definition}

\begin{remark}
When $A$ is a $C^{*}$-algebra, the $N$-$L^{p}$ Baum-Connes assembly map is indeed the Baum-Connes assembly map. In fact, in the context of $C^{*}$-algebras in \cite{Blackadar}, idempotents are homotopic to projections, and invertibles are homotopic to unitaries. And the norm of the projection or the unitary is no more than 1. 
\end{remark}

\begin{definition}
	Let $A$ and $B$ be $L^{p}$ operator algebras, and let $\omega: [1,\infty)\rightarrow [1,\infty)$ be a non-decreasing function. We say that $F^{N}: K^{N}_{i}(A)\rightarrow K^{N}_{j}(B)$ is $\omega$-surjective if for any integer $N\geq 1$ and $[y]\in K^{N}_{j}(B)$, there exists $[x]\in K^{\omega(N)}_{i}(A)$ such that
$$F^{\omega(N)}([x])=[y]\text{ in }K^{\omega(N)\cdot N}_{j}(B).$$
\end{definition}
\begin{remark}
By the proof of Theorem 5.17 in \cite{DC-Chung}, we know that if $\Gamma\curvearrowright X$ has finite dynamical complexity, then the $N$-$L^{p}$ Baum-Connes assembly map for $\Gamma\curvearrowright X$ is $\omega$-surjective, and the function $\omega$ depends on the dynamic asymptotic dimension $m$ and Mayer-Vietoris control pair $(\lambda,h)$. In addition, we may use the term controlled-surjective when we do not want to emphasize the function $\omega$.
\end{remark}

\begin{definition}
Let $A$ be a filtered $L^{p}$ operator algebra. For $0<\varepsilon<\frac{1}{20}$, $r>0$, and $N\geq 1$, we have a canonical group homomorphism
$$\iota^{N}_{*}: K^{\varepsilon,r,N}_{*}(A)\rightarrow K^{4N}_{*}(A),\quad[z]_{\varepsilon,r,N}\mapsto [\kappa_{*}(z)]_{4N}.$$
\end{definition}
Furthermore, the quantitative $L^{p}$ assembly maps are compatible with the $N$-$L^{p}$ assembly maps, namely, if $[z]$ is the element of $K^{N}_{*}\Big(B^{p}_{L}\big(P_{d}(\Gamma),A\big)^{\Gamma}\Big)$, then 
$$\mu^{36N,d}_{A,*}([z]_{36N})=\iota^{9N}_{*}\circ\mu^{\varepsilon,r,N,d}_{A,*}([z]_{N})\text{ in } K^{36N}_{*}(A\rtimes\Gamma).$$
\begin{proposition}\label{prop 3.24}
\item There exists a polynomial $\rho\geq 1$ with positive coefficients such that for any filtered $L^{p}$ operator algebra $A$, any $\varepsilon\in (0,\frac{1}{20\rho(N)})$, any $r>0$ and any $N\geq 1$, the following holds:
Let $[x], [x']$ be in $K^{\varepsilon,r,N}_{*}(A)$ such that $\iota^{N}_{*}([x])=\iota^{N}_{*}([x'])$ in $K^{4N}_{*}(A)$, there exists $r'\geq r$ such that 
$$[x]_{\rho(N)\varepsilon, r',33N}=[x']_{\rho(N)\varepsilon, r',33N}\text{ in }K^{\rho(N)\varepsilon, r',33N}_{*}(A).$$
\end{proposition}
\begin{proof}
(i) In the even case, let $(g_{t})_{t\in [0,1]}$ be a homotopy of $16N$-idempotents in $M_{n}(\widetilde{A})$ between $\kappa_{0}(x)$ and $\kappa_{0}(x')$. Then $G:=(g_{t})$ is a $16N$-idempotent in $C([0,1], M_{n}(\widetilde{A}))$. There exist $r'\geq r$ and $H:=(h_{t})\in C([0,1], M_{n}(\widetilde{A_{r'}}))$ such that $\Vert H-G\Vert<\frac{\varepsilon}{68N}$. In particular, we have $\Vert h_{0}-\kappa_{0}(x)\Vert<\frac{\varepsilon}{68N}$ and $\Vert h_{1}-\kappa_{0}(x')\Vert<\frac{\varepsilon}{68N}$. Then $h_{t}$ is an $(\varepsilon,r',17N)$-idempotent in $M_{n}(\widetilde{A})$ for each $t\in [0,1]$. Also
$$
\begin{aligned}
\Vert h_{0}-x\Vert<\Vert h_{0}-\kappa_{0}(x)\Vert+\Vert\kappa_{0}(x)-x\Vert\\
<\frac{\varepsilon}{68N}+\frac{2(N+1)\varepsilon}{(1-\sqrt{\varepsilon})(1-2\sqrt{\varepsilon})}\\
<6(N+1)\varepsilon\qquad\qquad\qquad\qquad
\end{aligned}
$$
and similarly $\Vert h_{1}-x'\Vert<6(N+1)\varepsilon$. Then $h_{0}$ and $x$ are $(\varepsilon',r',17N)$-homotopic, where $\varepsilon'=\varepsilon+\frac{1}{4}(6N+6)^{2}\varepsilon^{2}$, and similarly for $h_{1}$ and $x'$. Hence $[x]_{\varepsilon',r',17N}=[x']_{\varepsilon',r',17N}$.

(ii) In the odd case, let $(f_{t})_{t\in [0,1]}$ be a homotopy of $16N$-invertibles in $M_{n}(\widetilde{A})$ between $x$ and $x'$. $F=(f_{t})$ can be regarded as an invertible element in $C([0,1],M_{n}(\widetilde{A}))$. Then there exist $r'\geq r$ and $W\in C([0,1],M_{n}(\widetilde{A_{r'}}))$ such that 
$$\Vert W-F\Vert<\frac{1}{33N}(\varepsilon-\max\{\Vert xy-1\Vert,\Vert yx-1\Vert, \Vert x'y'-1\Vert, \Vert y'x'-1\Vert\}),$$
where $y$ is an $(\varepsilon,r,N)$-inverse for $x$, and $y'$ is an $(\varepsilon,r,N)$-inverse for $x'$. Then $W$ is an $(\varepsilon,r',33N)$-invertible in $C([0,1],M_{n}(\widetilde{A}))$, and we have a homotopy of $(\varepsilon,r',33N)$-invertibles $x\sim W_{0}\sim W_{1}\sim x'$.
\end{proof}

\subsection{Quantitative statements}
Oyono-Oyono established the connection between the Baum-Connes conjecture and the quantitative Baum-Connes conjecture in the article \cite{Oyono-2015}. In parallel, we will give the connection between the $L^p$ Baum-Connes conjecture and the quantitative $L^p$ Baum-Connes conjecture.

For a $\Gamma$-$L^{p}$ operator algebra $A$ and  positive numbers $d,d',r,r',\varepsilon,\varepsilon', N,N'$ with $d\leq d'$, $\varepsilon\leq\varepsilon'<\frac{1}{20}$, $r\leq r'$ and $1\leq N\leq N'$, let us consider the following statements:
\begin{itemize}
\item $QI_{A,*}(d,d',\varepsilon,r,N)$: for every $[x]\in K^{N}_{*}\Big(B^{p}_{L}\big(P_{d}(\Gamma),A\big)^{\Gamma}\Big)$, then 
$$\mu^{\varepsilon,r,N,d}_{A,*}([x])=0\text{ in }K^{\varepsilon,r,9N}_{*}(A\rtimes\Gamma)$$ 
implies that $i_{d,d',*}([x])=0$ in $K_{*}\Big(B^{p}_{L}\big(P_{d'}(\Gamma),A\big)^{\Gamma}\Big)$.

\item $QS_{A,*}(d,\varepsilon,\varepsilon',r,r',N,N')$: for every $[y]\in K^{\varepsilon,r,N}_{*}(A\rtimes\Gamma)$, there exists an element $[x]\in K^{N'}_{*}(B^{p}_{L}(P_{d}(\Gamma),A)^{\Gamma})$ such that 
$$\mu^{\varepsilon',r',N',d}_{A,*}([x])=\iota^{\varepsilon',r',9N'}_{*}([y])\text{ in }K^{\varepsilon',r',9N'}_{*}(A\rtimes\Gamma).$$
\end{itemize}

Using equation \ref{eq 1} and Proposition \ref{prop 2.11}, we get the following proposition:
\begin{proposition}
Let $\Gamma$ be a finitely generated group, and let $A$ be a $\Gamma$-$L^{p}$ operator algebra. For a positive number $\varepsilon$ with $\varepsilon<\frac{1}{20}$:

$\mathrm(i)$ Assume that for any $r>0$, $N\geq 1$ and $d>0$, there exists $d'\geq d$ such that $QI_{A,*}(d,d',\varepsilon,r,N)$ is satisfied. Then $\mu_{A,*}$ is injective.

$\mathrm(ii)$ Assume that for any $r>0$ and $N\geq 1$, there exist positive numbers $\varepsilon'$, $d$, $r'$ and $N'$ with $\varepsilon\leq\varepsilon'<\frac{1}{20}$, $r\leq r'$, $N\leq N'$ and $d>0$ such that $QS_{A,*}(d,\varepsilon,\varepsilon',r,r',N,N')$ is true. Then $\mu_{A,*}$ is surjective.
\end{proposition}

The following results construct the connection between quantitative injectivity (resp. surjectivity) and injectivity (resp. surjectivity) of the $L^{p}$ Baum-Connes assembly map.

\begin{theorem}\label{the 2}
Let $\Gamma$ be a discrete group, and let $A$ be a $\Gamma$-$L^{p}$ operator algebra. Then the following two statements are equivalent:

$\mathrm(i)$ $\mu_{\ell^{\infty}(\mathbb{N},\mathscr{K}(\ell^{p})\otimes A),*}$ is injective.

$\mathrm(ii)$ For $0<\varepsilon<\frac{1}{20}$, $r>0$, $N\geq 1$ and $d>0$, there exists $d'\geq d$ such that $QI_{A,*}(d,d',\varepsilon,r,N)$ holds.

\end{theorem}
\begin{proof}
The proof relies on  Proposition \ref{prop 3}, which will be proved later. Suppose (ii) holds. Let $[x]$ be in $K_{*}\Big(B^{p}_{L}\big(P_{d}(\Gamma), \ell^{\infty}(\mathbb{N},\mathscr{K}(\ell^{p})\otimes A)\big)^{\Gamma}\Big)$ for some $d>0$ such that
$$\mu^{d}_{\ell^{\infty}(\mathbb{N},\mathscr{K}(\ell^{p})\otimes A),*}([x])=0\text{ in }K_{*}\big(\ell^{\infty}(\mathbb{N},\mathscr{K}(\ell^{p})\otimes A)\rtimes\Gamma\big).$$
Then there exists $N'\geq1$ such that $[z]\in K^{N'}_{*}\big(\ell^{\infty}(\mathbb{N},\mathscr{K}(\ell^{p})\otimes A)\rtimes\Gamma\big)$, thus $\mu^{\varepsilon',r',N',d}_{\ell^{\infty}(\mathbb{N},\mathscr{K}(\ell^{p})\otimes A),*}([x])$ is an element of $K^{\varepsilon',r',9N'}_{*}\big(\ell^{\infty}(\mathbb{N},\mathscr{K}(\ell^{p})\otimes A)\rtimes\Gamma\big)$. By equation \ref{eq 1}, we obtain that 
$$\iota_{*}\big(\mu^{\varepsilon',r',N',d}_{\ell^{\infty}(\mathbb{N},\mathscr{K}(\ell^{p})\otimes A),*}([x])\big)=0\text{ for any }\varepsilon\in (0,\frac{1}{20}).$$
Hence, by Proposition \ref{prop 2.12} (ii) and equation \ref{eq 2}, there exist $\varepsilon\geq\varepsilon'$, $r\geq r'$ and $N\geq N'$ such that
 $$\mu^{\varepsilon,r,N,d}_{\ell^{\infty}(\mathbb{N},\mathscr{K}(\ell^{p})\otimes A),*}([x])=0\text{ in }K^{\varepsilon,r,9N}_{*}\big(\ell^{\infty}(\mathbb{N},\mathscr{K}(\ell^{p})\otimes A)\rtimes\Gamma\big).$$
According to Proposition \ref{prop 3}, we have an isomorphism
\begin{equation}\label{eq 3}
K_{*}\Big(B^{p}_{L}\big(P_{d}(\Gamma),\ell^{\infty}(\mathbb{N},\mathscr{K}(\ell^{p})\otimes A)\big)^{\Gamma}\Big)\xrightarrow{\cong}K_{*}\Big(B^{p}_{L}\big(P_{d}(\Gamma),A\big)^{\Gamma}\Big)^{\mathbb{N}}
\end{equation}
induced on the $j$-th factor by the projection $\ell^{\infty}(\mathbb{N},\mathscr{K}(\ell^{p})\otimes A)\rightarrow \mathscr{K}(\ell^{p})\otimes A$ and up to the Morita equivalence
\begin{equation}
K_{*}\Big(B^{p}_{L}\big(P_{d}(\Gamma), A\big)^{\Gamma}\Big)\cong K_{*}\Big(B^{p}_{L}\big(P_{d}(\Gamma),\mathscr{K}(\ell^{p})\otimes A\big)^{\Gamma}\Big).
\end{equation}
 Assume that $([x_{m}])_{m\in\mathbb{N}}$ is the element in $K_{*}\Big(B^{p}_{L}\big(P_{d}(\Gamma),A\big)^{\Gamma}\Big)^{\mathbb{N}}$ corresponding to $[x]$ under this identification, and let $d'\geq d$ be a positive number such that $QI_{A,*}(d,d',\varepsilon,r,N)$ holds. By naturality of the quantitative $L^{p}$ assembly maps, we get that
$$\mu^{\varepsilon,r,N,d}_{A,*}([x_{m}])=0\text{ in }K^{\varepsilon,r,9N}_{*}\Big(B^{p}_{L}\big(P_{d}(\Gamma),A\big)^{\Gamma}\Big),$$
which implies that $i_{d,d',*}([x_{m}])=0$ in $K_{*}\Big(B^{p}_{L}\big(P_{d'}(\Gamma),A\big)^{\Gamma}\Big)$ for each integer $m$. Finally, using equation \ref{eq 3}, we obtain that 
$$i_{d,d',*}([x])=0\text{ in }K_{*}\Big(B^{p}_{L}\big(P_{d'}(\Gamma),\ell^{\infty}(\mathbb{N},\mathscr{K}(\ell^{p})\otimes A)\big)^{\Gamma}\Big).$$
Hence $\mu_{\ell^{\infty}(\mathbb{N},\mathscr{K}(\ell^{p})\otimes A),*}$ is injective. Thus (ii) implies (i).

Suppose (ii) is false. In the even case, there exist $\varepsilon$ in $(0,\frac{1}{20})$, $r>0$, $N\geq 1$ and $d>0$ such that for all $d'\geq d$, the statement $QI_{A,0}(d,d',\varepsilon,r,N)$ does not hold. So it suffices to prove that $\mu_{\ell^{\infty}(\mathbb{N},\mathscr{K}(\ell^{p})\otimes A),0}$ is not injective. Let $(d_{m})_{m\in\mathbb{N}}$ be an increasing and unbounded sequence of positive numbers such that $d_{m}\geq d$ for all $m\in\mathbb{N}$. For each positive integer $m$, let $[x_{m}]$ be in $K^{N}_{0}\Big(B^{p}_{L}\big(P_{d}(\Gamma), A\big)^{\Gamma}\Big)$ such that 
$$\mu^{\varepsilon,r,N, d}_{A,0}([x_{m}])=0\text{ in }K^{\varepsilon,r,9N}_{0}(A\rtimes\Gamma)$$
but 
$$i_{d,d_{i},0}([x_{m}])\neq 0\text{ in }K_{0}\big(B^{p}_{L}(P_{d_{m}}(\Gamma),A)^{\Gamma}\big).$$
Assume that $[x]$ is the element in $K^{N}_{0}\Big(B^{p}_{L}\big(P_{d}(\Gamma),\ell^{\infty}(\mathbb{N},\mathscr{K}(\ell^{p})\otimes A)\big)^{\Gamma}\Big)$ corresponding to $([x_{m}])_{m\in\mathbb{N}}$ under the identification of equation \ref{eq 3}. Let $(e_{m})_{m\in\mathbb{N}}$ be a family of $(\varepsilon,r,9N)$-idempotents with $e_{m}$ in $M_{n_{k}}(\widetilde{A\rtimes\Gamma})$ for some $n_{k}$ such that 
$$\mu^{\varepsilon,r,N,d}_{\ell^{\infty}(\mathbb{N},\mathscr{K}(\ell^{p})\otimes A),0}([x])=[(e_{m})_{m\in\mathbb{N}}]_{\varepsilon,r,9N}\text{ in }K^{\varepsilon,r,9N}_{0}\big(\ell^{\infty}(\mathbb{N},\mathscr{K}(\ell^{p})\otimes A)\rtimes\Gamma\big).$$
By naturality of $\mu^{\varepsilon,r,N,d}_{A,0}$, we know that $[e_{m}]_{\varepsilon,r,9N}=0$ in $K^{\varepsilon,r,N}_{0}(A\rtimes\Gamma)$ for all integers $m$, hence
$$\iota_{0}([(e_{m})_{m\in\mathbb{N}}]_{\varepsilon,r,9N})=0\text{ in }K_{0}\big(\ell^{\infty}(\mathbb{N},\mathscr{K}(\ell^{p})\otimes A)\rtimes\Gamma\big).$$
This gives $\mu^{d}_{\ell^{\infty}(\mathbb{N},\mathscr{K}(\ell^{p})\otimes A),0}([x])=\iota_{0}\circ\mu^{\varepsilon,r,N,d}_{\ell^{\infty}(\mathbb{N},\mathscr{K}(\ell^{p})\otimes A),0}([x])=0$. For each positive integer $m$, $i_{d,d_{m},0}([x_{m}])\neq 0$ implies $i_{d,d_{m},0}([x])\neq 0$, thus we see that $\mu_{\ell^{\infty}(\mathbb{N},\mathscr{K}(\ell^{p})\otimes A),0}$ is not injective, hence (i) is false. In the odd case, we have a similar proof.
\end{proof}

\begin{theorem}
Let $\Gamma$ be a discrete group. Assume that for any $\Gamma$-$L^{p}$ operator algebra $A$, there exists a polynomial $\rho\geq 1$ with positive coefficients such that for any $\varepsilon$ in $(0,\frac{1}{20\rho(N)})$, $r>0$ and $N\geq 1$, there exist $r'\geq r$, $N'\geq N$ and $d>0$ such that $QS_{A,*}(d,\varepsilon,\rho(N)\varepsilon,r,r',N,N')$ holds. Then $\mu_{\ell^{\infty}(\mathbb{N},\mathscr{K}(\ell^{p})\otimes A),*}$ is surjective.

\end{theorem}
\begin{proof}
The proof relies on Proposition \ref{prop 3}, which will be proved later. Let $\rho$ be as in Proposition \ref{prop 2.11}. Suppose the statement $QS_{A,*}(d,\varepsilon,\rho(N)\varepsilon,r,r',N,N')$ holds.  Let $[z]$ be the element in  $K_{*}\big(\ell^{\infty}(\mathbb{N},\mathscr{K}(\ell^{p})\otimes A)\rtimes\Gamma\big)$ and let $[y]$ be in $K^{\varepsilon,r', N'}_{*}\big(\ell^{\infty}(\mathbb{N},\mathscr{K}(\ell^{p})\otimes A)\rtimes\Gamma\big)$ such that $\iota_{*}([y])=[z]$, with $\varepsilon\in (0,\frac{1}{20\rho(N)})$, $r>0$ and $N\geq 1$. Let $[y_{i}]$ be the image of $[y]$ under the composition
\begin{equation}\label{eq 5}
K^{\varepsilon,r, N}_{*}\big(\ell^{\infty}(\mathbb{N},\mathscr{K}(\ell^{p})\otimes A)\rtimes\Gamma\big)\rightarrow K^{\varepsilon,r, N}_{*}\big(\mathscr{K}(\ell^{p})\otimes A\rtimes\Gamma\big)\xrightarrow{\cong}K^{\varepsilon,r, N}_{*}(A\rtimes\Gamma),
\end{equation}
where the first map is induced on the $j$-th factor by the projection
$$\ell^{\infty}(\mathbb{N},\mathscr{K}(\ell^{p})\otimes A)\rightarrow \mathscr{K}(\ell^{p})\otimes A$$
and the second map is the Morita equivalence of Proposition \ref{prop2} and Proposition \ref{prop 2.17}. Let $d$, $r'$ and $N'$ be positive numbers with $r'\geq r$ and $N'\geq N$
such that $QS_{A,*}(d,\varepsilon,\rho(N)\varepsilon,r,r',N,N')$  holds. Then for each positive integer $m$, there exists $[x_{m}]$ in $K^{N'}_{*}\Big(B^{p}_{L}\big(P_{d}(\Gamma), A\big)^{\Gamma}\Big)$ such that 
$$\mu^{\rho(N)\varepsilon,r',N',d}_{A,*}([x_{m}])=\iota^{\rho(N)\varepsilon,r',9N'}_{*}([y_{m}])\text{ in }K^{\rho(N)\varepsilon,r',9N'}_{*}(A\rtimes\Gamma).$$
Let $[x]$ be the element of $K^{N'}_{*}\Big(B^{p}_{L}\big(P_{d}(\Gamma), \ell^{\infty}(\mathbb{N},\mathscr{K}(\ell^{p})\otimes A)\big)^{\Gamma}\Big)$ corresponding to $([x_{m}])_{m\in\mathbb{N}}$ under the identification of equation \ref{eq 3}. By naturality of the quantitative $L^{p}$ assembly maps, we get that 
$$\mu^{\rho(N)\varepsilon,r',N',d}_{\ell^{\infty}(\mathbb{N},\mathscr{K}(\ell^{p})\otimes A),*}([x])=\iota^{\rho(N)\varepsilon,r',9N'}_{*}([y])$$
in $K^{\rho(N)\varepsilon,r',9N'}_{*}\big(\ell^{\infty}(\mathbb{N},\mathscr{K}(\ell^{p})\otimes A)\rtimes\Gamma\big)$. Hence, we conclude that
$$\mu^{d}_{\ell^{\infty}(\mathbb{N},\mathscr{K}(\ell^{p})\otimes A),*}([x])=\iota_{*}([y])=[z],$$
and therefore $\mu_{\ell^{\infty}(\mathbb{N},\mathscr{K}(\ell^{p})\otimes A),*}$ is surjective.
\end{proof}

The next theorem gives the connection between controlled-surjectivity of the $\mathscr{N}$-$L^{p}$ Baum-Connes assembly map and quantitative surjectivity.

\begin{theorem}\label{the 3}
Let $\Gamma$ be a discrete group. Assume that for any $\Gamma$-$L^{p}$ operator algebra $A$ and any positive integer $\mathscr{N}$, there exists a non-decreasing function $\omega: [1,\infty)\rightarrow [1,\infty)$  such that $\mu^{\mathscr{N}}_{\ell^{\infty}(\mathbb{N},\mathscr{K}(\ell^{p})\otimes A),*}$ is $\omega$-surjective. Then for some polynomial $\rho\geq 1$ with positive coefficients and for any  $\varepsilon$ in $(0,\frac{1}{20\rho(9N\omega(4N))})$, $r>0$ and $N\geq 1$, there exist $r'\geq r$, $N'\geq N$ and $d>0$ such that $QS_{A,*}(d,\varepsilon,\rho(9N\omega(4N))\varepsilon,r,r',N,N')$ holds.
\end{theorem}
\begin{proof}
Assume that this statement does not hold. Then there exist 
\begin{itemize}
\item $\varepsilon$ in $(0,\frac{1}{20\rho(N)})$, $r>0$ and $N\geq 1$,

\item an unbounded increasing  sequence $(r_{m})_{m\in\mathbb{N}}$ with $r_{m}\geq r$, 

\item an unbounded increasing  sequence $(N_{m})_{m\in\mathbb{N}}$ with $N_{m}\geq N$,

\item an unbounded increasing  sequence $(d_{m})_{m\in\mathbb{N}}$ with $d_{m}>0$,

\item an element $[y_{m}]$ in $K^{\varepsilon,r,N}_{*}(A\rtimes\Gamma)$,
\end{itemize}
such that for each $m\in\mathbb{N}$ and any $[x_{m}]$ in $K^{N_{m}}_{*}\big(B^{p}_{L}(P_{d_{i}}(\Gamma), A)^{\Gamma}\big)$,
$$\iota^{\rho(9N\omega(4N))\varepsilon,r_{m},9N_{m}}_{*}([y_{m}])\neq\mu^{\rho(9N\omega(4N))\varepsilon,r_{m},N_{m},d_{m}}_{A,*}([x_{m}])$$ 
$\text{ in }K^{\rho(9N\omega(4N))\varepsilon,r_{m},9N_{m}}_{*}(A\rtimes\Gamma)$. According to equation \ref{eq 5}, there exists $$[y]\in K^{\varepsilon,r,N}_{*}\big(\ell^{\infty}(\mathbb{N},\mathscr{K}(\ell^{p})\otimes A)\rtimes\Gamma\big)$$ such that for every positive integer $m$, the image of $[y]$ is $[y_{m}]$. Since $\mu^{\mathscr{N}}_{\ell^{\infty}(\mathbb{N},\mathscr{K}(\ell^{p})\otimes A),*}$ is $\omega$-surjective, then for some $d'>0$ there exists $[x]$ in $K^{\omega(4N)}_{*}\Big(B^{p}_{L}\big(P_{d'}(\Gamma),\ell^{\infty}(\mathbb{N},\mathscr{K}(\ell^{p})\otimes A)\big)^{\Gamma}\Big)$ such that 
$$\iota^{N}_{*}([y])=\mu^{\omega(4N),d'}_{\ell^{\infty}(\mathbb{N},\mathscr{K}(\ell^{p})\otimes A),*}([x])\text{ in }K^{\omega(4N)\cdot 4N}_{*}\big(\ell^{\infty}(\mathbb{N},\mathscr{K}(\ell^{p})\otimes A)\rtimes\Gamma\big).$$
Since the quantitative $L^{p}$ assembly maps are compatible with the $\omega(4N)$-$L^{p}$ assembly maps, we get that
$$\mu^{4N_{1},d'}_{\ell^{\infty}(\mathbb{N},\mathscr{K}(\ell^{p})\otimes A),*}([x]_{4N_{1}})=\iota^{N_{1}}_{*}\circ\mu^{\varepsilon,r,\omega(4N),d'}_{\ell^{\infty}(\mathbb{N},\mathscr{K}(\ell^{p})\otimes A),*}([x]_{\omega(4N)}),$$
where $N_{1}=\max\{\omega(4N)\cdot N,9\omega(4N)\}$.
We now apply Proposition \ref{prop 3.24} and conclude that there exists $r'\geq r$ such that
$$\iota^{\rho(9N\omega(4N))\varepsilon,r',33N_{1}}_{*}\circ\mu^{\varepsilon,r,\omega(4N),d'}_{\ell^{\infty}(\mathbb{N},\mathscr{K}(\ell^{p})\otimes A),*}([x])=\iota^{\rho(9N\omega(4N))\varepsilon,r',33N_{1}}_{*}([y]).$$
However, if we choose $m$ such that $r_{m}\geq r'$, $N_{m}\geq 33N_{1}$ and $d_{m}\geq d'$, using naturality of the $L^{p}$ assembly map and equation \ref{eq 2}, we obtain that 
$$\iota^{\rho(9N\omega(4N))\varepsilon,r_{m},9N_{m}}_{*}([y_{m}])=\mu^{\rho(9N\omega(4N))\varepsilon,r_{m},N_{m},d_{m}}_{A,*}([x_{m}]),$$
which contradicts our assumption.
\end{proof}
In the proof of $(i)$ implies $(ii)$ of Theorem \ref{the 2} and Theorem \ref{the 3}, replacing the algebra \\$\ell^{\infty}(\mathbb{N},\mathscr{K}(\ell^{p})\otimes A)$ by $\prod\limits_{i\in\mathbb{N}}(\mathscr{K}(\ell^{p})\otimes A_{i})$ for a family of $\Gamma$-$L^{p}$ operator algebras $(A_{i})_{i\in\mathbb{N}}$ , we can obtain the following theorem.

\begin{theorem}
Let $\Gamma$ be a discrete group.

$\mathrm(i)$ Assume that for any $\Gamma$-$L^{p}$ operator algebra $A$, the $L^{p}$ Baum-Connes assembly map $\mu_{A,*}$ is injective. Then for  $0<\varepsilon<\frac{1}{20}$, $r>0$, $N\geq 1$ and $d>0$, there exists $d'\geq d$ such that $QI_{A,*}(d,d',\varepsilon,r,N)$ holds.

$\mathrm(ii)$ Assume that for any $\Gamma$-$L^{p}$ operator algebra $A$ and for any integer $\mathscr{N}$, there exists a non-decreasing function $\omega: [1,\infty)\rightarrow [1,\infty)$  such that the $\mathscr{N}$-$L^{p}$ Baum-Connes assembly map $\mu^{\mathscr{N}}_{A,*}$ is $\omega$-surjective. Then for some polynomial $\rho\geq 1$ with positive coefficients and for any $\varepsilon$ in $(0,\frac{1}{20\rho(9N\omega(4N))})$, $r>0$ and $N\geq 1$, there exist $d>0$, $r'\geq r$ and $N'\geq N$ such that \\$QS_{A,*}(d,\varepsilon,\rho(9N\omega(4N))\varepsilon,r,r',N,N')$ holds.
\end{theorem}
\begin{remark}
To complete the proof of Theorem \ref{the 2} and Theorem \ref{the 3}, we need Proposition \ref{prop 3} which is based on a couple of lemmas.
\end{remark}

\begin{lemma}\label{lemma 3.16}
Let $A$ be a unital $L^{p}$ operator algebra. There exists a map $\varphi:(0,\infty)\rightarrow (0,\infty)$ such that:
\begin{itemize}
\item If $e$ and $f$ are homotopic idempotents in $M_{n}(A)$, then there exist $k,N\in\mathbb{N}$ with $n+k\leq N$, and a homotopy of idempotents $(E_{t})_{t\in [0,1]}$ in $M_{N}(A)$ between $diag(e,I_{k},0)$ and $diag(f,I_{k},0)$ such that $\Vert E_{t}-E_{s}\Vert\leq\varepsilon$ when $\vert s-t\vert\leq\varphi(\varepsilon)$ for any $\varepsilon >0$ and any $s, t\in [0,1]$.

\item If $u$ and $v$ are homotopic invertibles in $GL_{n}(A)$, then there exist an integer $k$ and a homotopy $(U_{t})_{t\in[0,1]}$ in $GL_{n+k}(A)$ between $diag(u, I_{k})$ and $diag(v,I_{k})$ such that $\Vert U_{s}-U_{t}\Vert\leq\varepsilon$ when $\vert s-t\vert\leq\phi(\varepsilon)$ for any $\varepsilon>0$ and any $s, t\in [0,1]$.
\end{itemize}
\end{lemma}
\begin{proof}
Let us prove the property in the case of idempotents, the case of invertibles being similar.
Without loss of generality, we suppose $n=1$.

(i) Recall from proposition 4.3.3 and proposition 3.4.3 in \cite{Blackadar} that if $e$ and $f$ are idempotents in $A$, and there exists $0<\delta<\frac{1}{\Vert 2e-1\Vert}$ such that $\Vert e-f\Vert\leq\delta$, then $f=z^{-1}ez$ for some invertible $z$ in $A$ with $\Vert z-1\Vert<1$. Hence there exists $a\in A$ with $\Vert a\Vert<log 2$ such that $z=exp(a)$. Considering the homotopy $(e_{t})_{t\in[0,1]}=(exp(ta)\cdot e\cdot exp(-ta))_{t\in[0,1]}$ between $e$ and $f$, we see that there exists a map $\varphi_{1}: (0,\infty)\rightarrow (0,\infty)$ such that $\Vert e_{s}-e_{t}\Vert\leq\varepsilon$ when $\vert s-t\vert\leq\varphi_{1}(\varepsilon)$ for any $\varepsilon>0$ and any $s, t\in [0,1]$.

(ii) For $t\in [0,1]$, let $c_{t}=cos\frac{\pi t}{2}$ and $s_{t}=sin\frac{\pi t}{2}$. Define
$$E_{t}=\begin{pmatrix}e &0\\0 &0\end{pmatrix}+\begin{pmatrix}c_{t} & -s_{t}\\s_{t} &c_{t}\end{pmatrix}\begin{pmatrix}1-e&0\\0&0\end{pmatrix}\begin{pmatrix}c_{t} & s_{t}\\-s_{t} &c_{t}\end{pmatrix}$$
in $M_{2}(A)$. Then we know that $(E_{t})_{t\in[0,1]}$ is a homotopy of idempotents between $diag(1,0)$ and $diag(e,1-e)$. Also, there exists a map $\varphi_{2}: (0,\infty)\rightarrow (0,\infty)$ such that $\Vert E_{s}-E_{t}\Vert\leq\varepsilon$ when $\vert s-t\vert\leq\varphi_{2}(\varepsilon)$ for any $\varepsilon>0$ and any $s, t\in [0,1]$.

(iii) In the general case, let $(e_{t})_{t\in [0,1]}$ be a homotopy of idempotents between $e$ and $f$, and let $0=t_{0}<t_{1}<\cdots<t_{k}=1$ be such that 
$$\Vert e_{t_{i}}-e_{t_{i-1}}\Vert\leq\delta, \text{ for }i=1,\cdots,k.$$ 
Then we have the following sequence of homotopies of idempotents in $M_{2k+1}(A)$ in which the first and last homotopies are conjugated by some permutation matrices:

$h_{0}\overset{h^{0}_{t}}{\sim} h_{1}\overset{h^{1}_{t}}{\sim} h_{2}\overset{h^{2}_{t}}{\sim} h_{3}\overset{h^{3}_{t}}{\sim} h_{4}\overset{h^{4}_{t}}{\sim} h_{5}$, where

$h_{0}=diag(e_{t_{0}}, I_{k},0_{k})$,

$h_{1}=diag(e_{t_{0}},1,0,\cdots,1,0)$,

$h_{2}=diag(e_{t_{0}},1-e_{t_{1}},e_{t_{1}},\cdots,1-e_{t_{k}},e_{t_{k}})$,

$h_{3}=diag(e_{t_{0}},1-e_{t_{0}},e_{t_{1}},1-e_{t_{1}},\cdots,e_{t_{k-1}},1-e_{t_{k-1}},e_{t_{k}})$,

$h_{4}=diag(1,0,\cdots,1,0,e_{t_{k}})$,

$h_{5}=diag(e_{t_{k}},I_{k},0_{k})$.

If we let $\varphi=min\{\varphi_{1},\varphi_{2}\}$, then the result is obtained from case (i) and case (ii). Indeed, the fact that $\Vert h_{3}-h_{2}\Vert\leq\delta$ implies that for every $m\in\{0,4\}$, there are homotopies $(h^{m}_{t})_{t\in[0,1]}$ between $h_{m}$ and $h_{m+1}$ such that $\Vert h^{m}_{s}-h^{m}_{t}\Vert\leq\varepsilon$ when $\vert s-t\vert\leq\varphi(\varepsilon)$ for any $\varepsilon>0$ and any $s, t\in [0,1]$.

\end{proof}

In the next lemma, the injectivity of $\Phi^{\mathcal{A}}_{*}$ follows immediately from the above Lemma \ref{lemma 3.16}, and $\Phi^{\mathcal{A}}_{*}$ is clearly surjective. Hence the following result is  obtained:
\begin{lemma}\label{lemma1} 
Let $\mathcal{A}=(A_{i})_{i\in I}$ be a family of unital $L^{p}$ operator algebras. Let 
$$\Phi^{\mathcal{A}}_{*}: K_{*}\Big(\prod\limits_{i\in I}\big( \mathscr{K}(\ell^{p})\otimes A_{i}\big)\Big)\longrightarrow\prod\limits_{i\in I} K_{*}\big(\mathscr{K}(\ell^{p})\otimes A_{i}\big)\cong\prod\limits_{i\in I}K_{*}(A_{i})$$ be the homomorphism induced on the $j$-th factor by the projection
$$\prod\limits_{i\in I}\big(\mathscr{K}(\ell^{p})\otimes A_{i}\big)\rightarrow \mathscr{K}(\ell^{p})\otimes A_{j}.$$
Then $\Phi^{\mathcal{A}}_{*}$ is an isomorphism.
\end{lemma}
\begin{remark}
Observe that $\mathscr{K}(\ell^{p})\otimes\mathscr{K}(\ell^{p})\otimes A_{i}$ is isometrically isomorphic to $\mathscr{K}(\ell^{p})\otimes A_{i}$ for each $i\in\mathbb{N}$, thus $\Phi^{\mathcal{A}}_{*}$ is an isometric isomorphism.
\end{remark}
As a consequence of this lemma, we have the following important proposition:
\begin{proposition}\label{prop 3}
Let $\Gamma$ be a discrete group and let $\mathcal{A}=(A_{i})_{i\in\mathbb{N}}$ be a family of $\Gamma$-$L^{p}$ operator algebras. Suppose $A_{i}\otimes\mathscr{K}(\ell^{p})$ is equipped with the diagonal action, the action of $\Gamma$ on $\mathscr{K}(\ell^{p})$ is trivial. Let 
$$\Phi^{\Gamma,\mathcal{A}}_{*}:K_{*}\bigg(B^{p}_{L}\Big(P_{d}(\Gamma),\prod\limits_{i\in I}\big(\mathscr{K}(\ell^{p})\otimes A_{i}\big)\Big)^{\Gamma}\bigg)\longrightarrow \prod\limits_{i\in I}K_{*}\Big(B^{p}_{L}\big(P_{d}(\Gamma),\mathscr{K}(\ell^{p})\otimes A_{i}\big)^{\Gamma}\Big)\cong\prod\limits_{i\in I}K_{*}\Big(B^{p}_{L}\big(P_{d}(\Gamma),A_{i}\big)^{\Gamma}\Big)$$
be the homomorphism induced on the $j$-th factor by the projection
$$\prod\limits_{i\in I}\Big(\mathscr{K}(\ell^{p})\otimes A_{i}\Big)\rightarrow \mathscr{K}(\ell^{p})\otimes A_{j}.$$
Then $\Phi^{\Gamma,\mathcal{A}}_{*}$ is an isomorphism.
\end{proposition}

\begin{proof}
Put $B_{i}=\mathscr{K}(\ell^{p})\otimes A_{i}$, $i\in I$. For any locally compact space $X$ equipped with an action of $\Gamma$, we define
$$\Phi^{X}_{*}:K_{*}\big(B^{p}_{L}(X,\prod\limits_{i\in I}B_{i})^{\Gamma}\big)\rightarrow\prod\limits_{i\in I}K_{*}\big(B^{p}_{L}(X,B_{i})^{\Gamma}\big).$$
The homomorphism induced by the projection on the $j$-th factor is
$$\Phi^{X}_{j,*}:K_{*}\big(B^{p}_{L}(X,\prod\limits_{i\in I}B_{i})^{\Gamma}\big)\rightarrow K_{*}\big(B^{p}_{L}(X,B_{j})^{\Gamma}\big).$$
Let $Z_{0},\cdots,Z_{n}$ be the skeleton decomposition of $P_{d}(\Gamma)$, then $Z_{j}$ is a locally finite simplicial complex of dimension $j$, and endowed with a proper, cocompact and type preserving action of $\Gamma$.

Next, we prove that $\Phi^{Z_{j}}_{*}$ is an isomorphism by induction on $j$.

(i) For $j=0$, the $0$-skeleton $Z_{0}$ is a finite union of orbits, thus it suffices to prove that $\Phi^{\Gamma/F}_{*}$ is an isomorphism when $F$ is a finite subgroup of $\Gamma$. For any $\Gamma$-$L^{p}$ operator algebra $B$, Let $\chi_{0}$ be the charateristic map of ${F}$ in $\Gamma/F$, and let $\pi$ be a representation of $C_{0}(\Gamma/F)$ in $E_{d}$. Then $E_{d_{0}}=\pi(\chi_{0})\cdot E_{d}$ is stable by the action of group $F$ and by the endmorphism of a bounded linear operator $T$. The element restricted on  $E_{d_{0}}$ defines an element of $K_{*}\big(B^{p}_{L}(\mathbb{C},B)^{F}\big)$ and there is a natural restriction isomorphism
$$R^{B}_{F,\Gamma}:K_{*}\big(B^{p}_{L}(\Gamma/F,B)^{\Gamma}\big)\rightarrow K_{*}\big(B^{p}_{L}(\mathbb{C},B)^{F}\big)\cong K_{*}(B\rtimes F).$$
By naturality, we obtain the following commutative diagram:
$$\begin{CD}
K_{*}\big(B^{p}_{L}(\Gamma/F,\prod\limits_{i\in I}B_{i})^{\Gamma}\big)@>{\Phi^{\Gamma/F}_{j,*}}>>K_{*}\big(B^{p}_{L}(\Gamma/F, B_{j})^{\Gamma}\big)
\\@V{R^{\prod\limits_{i\in I}B_{i}}_{F,\Gamma}}VV@VV{R^{B_{j}}_{F,\Gamma}}V\\
K_{*}(\prod\limits_{i\in I}B_{i}\rtimes F)@>>>K_{*}(B_{j}\rtimes F),
\end{CD}$$
where the bottom row is induced by the homomorphism 
$$\prod\limits_{i\in I}B_{i}\rtimes F\rightarrow B_{j}\rtimes F$$ 
determined by the projection on the $j$-th factor $\prod\limits_{i\in I}B_{i}\rightarrow B_{j}$. Since $F$ is finite, we see that $\prod\limits_{i\in I}B_{i}\rtimes F\cong\big(\prod\limits_{i\in I}B_{i}\big)\rtimes F$. Applying Lemma \ref{lemma1}, we have an isomorphism
$$K_{*}\big(\big(\prod\limits_{i\in I}B_{i}\big)\rtimes F)\big)\cong K_{*}\big(\prod\limits_{i\in I}B_{i}\rtimes F\big)\rightarrow\prod\limits_{i\in I}K_{*}(B_{i}\rtimes F).$$
Hence $\Phi^{\Gamma/F}_{*}$ is an isomorphism.

(ii) Suppose $\Phi^{Z_{j-1}}_{*}$ is an isomorphism, and it remains to prove that $\Phi^{Z_{j}}_{*}$ is an isomorphism. The short exact sequence
 $$0\rightarrow C_{0}(Z_{j}\backslash Z_{j-1})\rightarrow C_{0}(Z_{j})\rightarrow C_{0}(Z_{j-1})\rightarrow 0$$
induces a natural long exact sequence
$$\rightarrow K_{*}\big(B^{p}_{L}(Z_{j-1},\cdot)^{\Gamma}\big)\rightarrow K_{*}\big(B^{p}_{L}(Z_{j},\cdot)^{\Gamma}\big)\rightarrow K_{*}\big(B^{p}_{L}(Z_{j}\backslash Z_{j-1},\cdot)^{\Gamma}\big)\rightarrow K_{*+1}\big(B^{p}_{L}(Z_{j-1},\cdot)^{\Gamma}\big)\rightarrow$$
and hence by naturality, we obtain a commutative diagram:
$$  \xymatrix@C=0.3cm@R=1cm{
   K_{*}\big(B^{p}_{L}(Z_{j-1},B)^{\Gamma}\big)\ar[r] \ar[d]_{\Phi^{Z_{j-1}}_{*}}& K_{*}\big(B^{p}_{L}(Z_{j},B)^{\Gamma}\big)\ar[r]\ar[d]_{\Phi^{Z_{j}}_{*}}& K_{*}\big(B^{p}_{L}(Z_{j}\backslash Z_{j-1},B)^{\Gamma}\big)\ar[r]\ar[d]_{\Phi^{Z_{j}\backslash Z_{j-1}}_{*}}& K_{*+1}\big(B^{p}_{L}(Z_{j-1},B)^{\Gamma}\big)\ar[d]_{\Phi^{Z_{j-1}}_{*+1}}\\
   \prod\limits_{i\in I}K_{*}\big(B^{p}_{L}(Z_{j-1})^{\Gamma}\big)\ar[r]& \prod\limits_{i\in I}K_{*}\big(B^{p}_{L}(Z_{j})^{\Gamma}\big)\ar[r]& \prod\limits_{i\in I}K_{*}\big(B^{p}_{L}(Z_{j}\backslash Z_{j-1})^{\Gamma}\big)\ar[r]& \prod\limits_{i\in I}K_{*+1}\big(B^{p}_{L}(Z_{j-1})^{\Gamma}\big),
  }$$
where $\prod\limits_{i\in I}B_{i}$ and $\prod\limits_{i\in I}K_{*}\big(B^{p}_{L}(Z_{j},B_{i})^{\Gamma}\big)$  are denoted by $B$ and $\prod\limits_{i\in I}K_{*}\big(B^{p}_{L}(Z_{j})^{\Gamma}\big)$ respectively.
We denote by $I_{j}$ the interior of the standard $j$-simplex. Since the action of $\Gamma$ is type preserving, then 
$$Z_{j}\backslash Z_{j-1}\cong I_{j}\times C_{j},$$
where $C_{j}$ is the set of center of $j$-simplices of $Z_{j}$, $\Gamma$ acts trivially on $I_{j}$. Together with the Bott periodicity, we have a commutative diagram:
$$\begin{CD}
K_{*}\big(B^{p}_{L}(Z_{j}\backslash Z_{j-1},\prod\limits_{i\in I}B_{i})^{\Gamma}\big)@>>>K_{*+1}\big(B^{p}_{L}(C_{j},\prod\limits_{i\in I}B_{i})^{\Gamma}\big)
\\@V{\Phi^{Z_{j}\backslash Z_{j-1}}_{*}}VV@VV{\Phi^{C_{j}}_{*+1}}V\\
\prod\limits_{i\in I}K_{*}\big(B^{p}_{L}(Z_{j}\backslash Z_{j-1},B_{i})^{\Gamma}\big)@>>>\prod\limits_{i\in I}K_{*+1}\big(B^{p}_{L}(C_{j},B_{i})^{\Gamma}\big).
\end{CD}$$
Finally, $\Phi^{C_{j}}_{*}$ is an isomorphism obtained from case (i), and thus $\Phi^{Z_{j}\backslash Z_{j-1}}_{*}$ is an isomorphism.
According to the induction hypothesis and the five lemma, we know that $\Phi^{Z_{j}}_{*}$ is an isomorphism.

\end{proof}

\section{Persistence approximation property}
In this section, we introduce the persistence approximation property for filtered $L^{p}$ operator algebras. In the case of a reduced crossed product of an $L^{p}$ operator algebra by a finitely generated group, we find a sufficient condition for the persistence approximation property.

Let $A$ be a filtered $L^{p}$ operator algebra. Applying Proposition \ref{prop 2.12} (i), we see that for any $\varepsilon\in (0,\frac{1}{20})$ and any $N\geq 1$, there exists a surjective map
$$\lim\limits_{r>0}K_{*}^{\varepsilon,r,N}(A)\rightarrow K^{N}_{*}(A)$$
induced by a family of relaxation of control maps $(\iota_{*})_{r>0}$. Moreover, if $\varepsilon>0$ is small enough, then for any $r>0$, any $N\geq 1$ and any $[x]\in K^{\varepsilon,r,N}_{*}(A)$, there exist positive numbers $\varepsilon'\in [\varepsilon,\frac{1}{20})$ independent of $x$ and $A$, $r'\geq r$ and $N'\geq N$ such that
$$\iota_{*}([x])=0\text{ in }K_{*}(A) \Rightarrow\iota^{\varepsilon', r',N'}_{*}([x])=0\text{ in } K^{\varepsilon', r', N'}_{*}(A).$$
However, we may wonder whether this $r'$ depends on $x$, in other words whether the family $\big(K^{\varepsilon,r,N}_{*}(A)\big)_{0<\varepsilon<\frac{1}{20},r>0,N\geq 1}$ has a persistence approximation for $K_{*}(A)$ in the following sense: for any sufficiently small $\varepsilon\in (0,\frac{1}{20})$, any $r>0$ and any $N\geq 1$, there exist $\varepsilon'\in [\varepsilon,\frac{1}{20})$, $r'\geq r$ and $N'\geq N$ such that for any $[x]\in K^{\varepsilon,r,N}_{*}(A)$, we have 
$$\iota^{\varepsilon',r',N'}_{*}([x])\neq 0\text{ in }K^{\varepsilon', r', N'}_{*}(A)\Rightarrow\iota_{*}([x])\neq 0\text{ in }K_{*}(A).$$
Therefore, we consider the following statement: For a filtered $L^{p}$ operator algebra $A$ and positive numbers $\varepsilon$, $r$ and $N\geq 1$, there exist $\varepsilon'$ in $[\varepsilon,\frac{1}{20})$, $r'\geq r$ and $N'\geq N$:

$\mathcal{PA_{*}}(A, \varepsilon,\varepsilon',r,r',N,N')$: $\text{ for any }[x]\in K^{\varepsilon,r,N}_{*}(A)$, $$\iota_{*}([x])=0\text{ in }K_{*}(A) \Rightarrow \iota^{\varepsilon',r',N'}_{*}([x])=0\text{ in } K^{\varepsilon',r',N'}_{*}(A).$$
\subsection{The case of crossed products.}
\begin{theorem}\label{theorem 4.1}
Let $\Gamma$ be a finitely generated group, and let $A$ be a $\Gamma$-$L^{p}$ operator algebra. Assume that
\begin{itemize}
\item $\Gamma$ admits a cocompact universal example for proper actions;
\item for any positive integer $\mathscr{N}$, there exists a non-decreasing function $\omega: [1,\infty)\rightarrow [1,\infty)$ such that the $\mathscr{N}$-$L^{p}$ Baum-Connes assembly map for $\Gamma$ with coefficients in 
$$\ell^{\infty}(\mathbb{N}, \mathscr{K}(\ell^{p})\otimes A)$$
is $\omega$-surjective;
\item the $L^{p}$ Baum-Connes assembly map for $\Gamma$ with coefficients in $A$ is injective.
\end{itemize}
Then for any $N\geq 1$, there exists a universal constant $\lambda_{PA}\geq 1$ such that for any $\varepsilon$ in $(0,\frac{1}{20\lambda_{PA}})$ and any $r>0$, there exist $r'\geq r$ and $N'\geq N$ such that $\mathcal{PA_{*}}(A\rtimes\Gamma, \varepsilon,\lambda_{PA}\varepsilon,r,r',N,N')$ holds.
\end{theorem}

\begin{remark}
Here, the constant $\lambda_{PA}$ does not depend on $r$, but on the positive integer $N$. 
\end{remark}

\begin{proof}
Let $A$ be a $\Gamma$-$L^{p}$ operator algebra, and let $\Gamma$ admit a cocompact universal example for proper actions. Assume that for every positive integer $\mathscr{N}$, there exists a non-decreasing function $\omega$ such that the $\mathscr{N}$-$L^{p}$ Baum-Connes assembly map with coefficients in $\ell^{\infty}(\mathbb{N}, \mathscr{K}(\ell^{p})\otimes A)$
is $\omega$-surjective and the $L^{p}$ Baum-Connes assembly map with coefficients in $A$ is injective, then there exist positive numbers $d$ and $d'$ with $d\leq d'$ such that the following two conditions are satisfied:
\begin{itemize}
\item for every $\mathscr{N}\in\mathbb{N}$ and any $[z]$ in $K^{\mathscr{N}}_{*}\big(\ell^{\infty}(\mathbb{N}, \mathscr{K}(\ell^{p})\otimes A\big)\rtimes\Gamma)$, there exists $[x]$ in\\ $K^{\omega(\mathscr{N})}_{*}\Big(B^{p}_{L}\big(P_{d}(\Gamma),\ell^{\infty}(\mathbb{N}, \mathscr{K}(\ell^{p})\otimes A)\big)^{\Gamma}\Big)$ such that 
$$\mu^{\omega(\mathscr{N}),d}_{\ell^{\infty}(\mathbb{N}, \mathscr{K}(\ell^{p})\otimes A),*}([x])=[z]\text{ in }K^{\omega(\mathscr{N})\cdot\mathscr{N}}_{*}\big(\ell^{\infty}(\mathbb{N}, \mathscr{K}(\ell^{p})\otimes A)\rtimes\Gamma\big).$$

\item for any $[x]$ in $K_{*}\Big(B^{p}_{L}\big(P_{d}(\Gamma),A\big)^{\Gamma}\Big)$ such that $\mu^{d}_{A,*}([x])=0$, we have
$$i_{d,d',*}([x])=0\text{ in }K_{*}\Big(B^{p}_{L}\big(P_{d'}(\Gamma),A\big)^{\Gamma}\Big),$$
where $i_{d,d',*}:K_{*}\Big(B^{p}_{L}\big(P_{d}(\Gamma),A\big)^{\Gamma}\Big)\rightarrow K_{*}\Big(B^{p}_{L}\big(P_{d'}(\Gamma),A\big)^{\Gamma}\Big)$is induced by the inclusion $P_{d}(\Gamma)\hookrightarrow P_{d'}(\Gamma)$.
\end{itemize}
Fix such $d$ and $d'$, and let $\rho$ be as in Proposition \ref{prop 3.24}, pick $(\lambda,h)$ as in Lemma \ref{lemma 2} and put $\lambda_{PA}=\rho(9\lambda_{N}\omega(4\lambda_{N}))$. Assume that there exists $N\geq 1$ such that this statement does not hold. Then there exist 
\begin{itemize}
\item $\varepsilon\in (0,\frac{1}{20\lambda_{PA}})$ and $r>0$,

\item an unbounded increasing sequence $(r_{i})_{i\in\mathbb{N}}$ with $r_{i}\geq r$,

\item an unbounded increasing sequence $(N_{i})_{i\in\mathbb{N}}$ with $N_{i}\geq N$,

\item a sequence of elements $([x_{i}])_{i\in\mathbb{N}}$ with $[x_{i}]\in K^{\varepsilon,r,N}_{*}(A\rtimes\Gamma)$,

\end{itemize}
such that, for each $i\in\mathbb{N}$,
$$\iota_{*}([x_{i}])=0\text{ in }K_{*}(A\rtimes\Gamma)$$
 and 
$$\iota^{\lambda_{PA}\varepsilon,r_{i},N_{i}}_{*}([x_{i}])\neq 0\text{ in }K^{\lambda_{PA}\varepsilon,r_{i},N_{i}}_{*}(A\rtimes\Gamma).$$
Since 
$$\ell^{\infty}(\mathbb{N}, \mathscr{K}(\ell^{p})\otimes A)\rtimes\Gamma_{h_{\varepsilon,N}r}=\ell^{\infty}(\mathbb{N}, \mathscr{K}(\ell^{p})\otimes A\rtimes\Gamma_{h_{\varepsilon,N}r})$$
and according to Lemma \ref{lemma 2}, there exists an element
$$[x]\in K^{\lambda_{N}\varepsilon,h_{\varepsilon,N}r,\lambda_{N}}_{*}\big(\ell^{\infty}(\mathbb{N}, \mathscr{K}(\ell^{p})\otimes A)\rtimes\Gamma\big)$$
that maps to $\iota^{\lambda_{N}\varepsilon,h_{\varepsilon,N}r,\lambda_{N}}_{*}([x_{i}])$, for all integers $i$ under the composition
$$K_{*}^{\lambda_{N}\varepsilon,h_{\varepsilon,N}r,\lambda_{N}}\big(\ell^{\infty}(\mathbb{N}, \mathscr{K}(\ell^{p})\otimes A)\rtimes\Gamma\big)\rightarrow K_{*}^{\lambda_{N}\varepsilon,h_{\varepsilon,N}r,\lambda_{N}}(\mathscr{K}(\ell^{p})\otimes A\rtimes\Gamma)
\xrightarrow{\cong} K^{\lambda_{N}\varepsilon,h_{\varepsilon, N}r,\lambda_{N}}_{*}(A\rtimes\Gamma),$$
where the first map is induced by the $j$-th projection
\begin{equation}\label{eq 6}
\ell^{\infty}(\mathbb{N}, \mathscr{K}(\ell^{p})\otimes A)\rightarrow \mathscr{K}(\ell^{p})\otimes A
\end{equation}
and the isomorphism is the Morita equivalence of Proposition \ref{prop2} and Proposition \ref{prop 2.17}. Note that $\iota^{\lambda_{N}}_{*}([x])$ is in $K^{4\lambda_{N}}_{*}\big(\ell^{\infty}(\mathbb{N}, \mathscr{K}(\ell^{p})\otimes A)\rtimes\Gamma\big)$. Let
$$[z]\in K^{\omega(4\lambda_{N})}_{*}\Big(B^{p}_{L}\big(P_{d}(\Gamma),\ell^{\infty}(\mathbb{N}, \mathscr{K}(\ell^{p})\otimes A)\big)^{\Gamma}\Big)$$
such that 
$$\mu^{\omega(4\lambda_{N}),d}_{\ell^{\infty}(\mathbb{N}, \mathscr{K}(\ell^{p})\otimes A),*}([z])=\iota^{\lambda_{N}}_{*}([x])\text{ in }K^{\omega(4\lambda_{N})\cdot 4\lambda_{N}}_{*}\big(\ell^{\infty}(\mathbb{N}, \mathscr{K}(\ell^{p})\otimes A)\rtimes\Gamma\big).$$
Since the quantitative $L^p$ assembly maps are compatible with the $\omega(4\lambda_{N})$-$L^{p}$ assembly maps, we obtain that
$$\mu^{4N_{1},d}_{\ell^{\infty}(\mathbb{N}, \mathscr{K}(\ell^{p})\otimes A),*}([z]_{4N_{1}})=\iota^{N_{1}}_{*}\circ\mu^{\lambda_{N}\varepsilon, h_{\varepsilon, N}r,\omega(4\lambda_{N}),d}_{\ell^{\infty}(\mathbb{N}, \mathscr{K}(\ell^{p})\otimes A),*}([z]_{\omega(4\lambda_{N})}),$$
where $N_{1}=\max\{\omega(4\lambda_{N})\cdot \lambda_{N},9\omega(4\lambda_{N})\}$.
However, according to Proposition \ref{prop 3.24}, there exists $R\geq h_{\varepsilon,N}r$ such that
$$\iota^{\lambda_{PA}\varepsilon,R,33N_{1}}_{*}([x])=\iota^{\lambda_{PA}\varepsilon,R,33N_{1}}_{*}\circ\mu^{\lambda_{N}\varepsilon,h_{\varepsilon, N}r,\omega(4\lambda_{N}),d}_{\ell^{\infty}(\mathbb{N}, \mathscr{K}(\ell^{p})\otimes A),*}([z]_{\omega(4\lambda_{N})})$$
 $$=\mu^{\lambda_{PA}\varepsilon, R,33N_{1},d}_{\ell^{\infty}(\mathbb{N}, \mathscr{K}(\ell^{p})\otimes A),*}([z]_{33N_{1}}).$$
By Proposition \ref{prop 3}, we have an isomorphism

\begin{equation}\label{eq 7}
K_{*}\Big(B^{p}_{L}\big(P_{d}(\Gamma),\ell^{\infty}(\mathbb{N}, \mathscr{K}(\ell^{p})\otimes A))^{\Gamma}\Big)\xrightarrow{\cong}\prod\limits_{j\in\mathbb{N}}K_{*}\Big(B^{p}_{L}\big(P_{d}(\Gamma),A\big)^{\Gamma}\Big)
\end{equation}
induced by the $j$-th projection in equation \ref{eq 6}. Let $([z_{j}])_{j\in\mathbb{N}}$ be the element of 
$$\prod\limits_{j\in\mathbb{N}}K_{*}\Big(B^{p}_{L}\big(P_{d}(\Gamma),A\big)^{\Gamma}\Big)$$
corresponding to $[z]$ under this identification.
Using the compatibility of the quantitative $L^p$ assembly maps with the usual ones, we obtain by naturality that $\mu^{d}_{A_{i},*}([z_{i}])=0$, for every $ i\in\mathbb{N} $ and  hence
$$i_{d,d',*}([z_{i}])=0\text{ in }K_{*}\Big(B^{p}_{L}\big(P_{d'}(\Gamma),A\big)^{\Gamma}\Big).$$
Using once more equation \ref{eq 7}, we deduce that 
$$i_{d,d',*}([z])=0\text{ in }K_{*}\Big(B^{p}_{L}\big(P_{d'}(\Gamma),\ell^{\infty}(\mathbb{N}, \mathscr{K}(\ell^{p})\otimes A)\big)^{\Gamma}\Big).$$
Let $(p_{t})_{t\in [0,1]}$ be a homotopy of idempotents (resp. invertibles) in $M_{n}(\widetilde{B})$ between $i_{d,d',*}([z])$ and $0$, then $P:=(p_{t})$ is an idempotent (resp. invertible) element in $C([0,1], M_{n}(\widetilde{B}))$, where $B=B^{p}_{L}\big(P_{d'}(\Gamma),\ell^{\infty}(\mathbb{N}, \mathscr{K}(\ell^{p})\otimes A)\big)^{\Gamma}$. Put $N'=\max\{33N_{1},\Vert P\Vert\}$. Since
$$\mu^{\lambda_{PA}\varepsilon,R,N',d}_{\ell^{\infty}(\mathbb{N}, \mathscr{K}(\ell^{p})\otimes A),*}([z])=\mu^{\lambda_{PA}\varepsilon,R,N',d'}_{\ell^{\infty}(\mathbb{N}, \mathscr{K}(\ell^{p})\otimes A),*}\circ i_{d,d',*}([z]),$$
then
$$\iota^{\lambda_{PA}\varepsilon,R,N'}_{*}([x])=0\text{ in }K^{\lambda_{PA}\varepsilon,R,N'}_{*}\big(\ell^{\infty}(\mathbb{N}, \mathscr{K}(\ell^{p})\otimes A)\rtimes\Gamma\big).$$
By naturality, we see that $\iota^{\lambda_{PA}\varepsilon,R,N'}_{*}([x_{i}])=0$ in $K^{\lambda_{PA}\varepsilon,R,N'}_{*}(A\rtimes\Gamma)$,  for all integers $i$. Picking an integer $i$ such that $r_{i}\geq R$ and $N_{i}\geq N'$, we have
$$\iota^{\lambda_{PA}\varepsilon,r_{i},N_{i}}_{*}([x_{i}])=\iota^{\lambda_{PA}\varepsilon,r_{i},N_{i}}_{*}\circ\iota^{\lambda_{PA}\varepsilon,R,N'}_{*}([x_{i}])=0,$$
which contradicts our assumption.
\end{proof}

For any $L^{p}$ operator algebra $A$, the $L^{p}$ Baum-Connes assembly map for $\Gamma$ with coefficients in $C_{0}(\Gamma,A)$ is an isomorphism
and $C_{0}(\Gamma,A)\rtimes\Gamma\cong A\otimes \mathscr{K}(\ell^{p}(\Gamma))$, hence by Theorem \ref{theorem 4.1}, we immediately obtain the following corollary.

\begin{corollary}\label{cor 1}
Let $\Gamma$ be a finitely generated group, and let $A$ be an $L^{p}$ operator algebra. Assume that
\begin{itemize}
\item $\Gamma$ admits a cocompact universal example for proper actions;
\item for any positive integer $\mathscr{N}$, there exists a non-decreasing function $\omega: [1,\infty)\rightarrow [1,\infty)$ such that the $\mathscr{N}$-$L^{p}$ Baum-Connes assembly map for $\Gamma$ with coefficients in 
$$\ell^{\infty}\big(\mathbb{N}, C_{0}(\Gamma,\mathscr{K}(\ell^{p})\otimes A)\big)$$
is $\omega$-surjective.
\end{itemize}
Then for any $N\geq 1$, there exists a universal constant $\lambda_{PA}\geq 1$ such that for any $\varepsilon$ in $(0,\frac{1}{20\lambda_{PA}})$ and any $r>0$, there exist $r'\geq r$ and $N'\geq N$ such that $\mathcal{PA_{*}}\big(A\otimes\mathscr{K}(\ell^{p}(\Gamma)), \varepsilon,\lambda_{PA}\varepsilon,r,r',N,N'\big)$ holds. 
\end{corollary}

In particular, if we put $A=\mathbb{C}$, we have the following conclusion.

\begin{proposition}
Let $\Gamma$ be a finitely generated group. Assume that
\begin{itemize}
\item $\Gamma$ admits a cocompact universal example for proper actions;
\item for any positive integer $\mathscr{N}$, there exists a non-decreasing function $\omega: [1,\infty)\rightarrow [1,\infty)$ such that the $\mathscr{N}$-$L^{p}$ Baum-Connes assembly map for $\Gamma$ with coefficients in  
$$\ell^{\infty}\big(\mathbb{N}, C_{0}(\Gamma,\mathscr{K}(\ell^{p})\big)$$
is $\omega$-surjective.
\end{itemize}
Then for any $N\geq 1$, there exists a universal constant $\lambda\geq 1$ such that for any $\varepsilon\in (0,\frac{1}{20\lambda})$ and any $r>0$, there exist $R\geq r$ and $N'\geq N$ such that the following holds:
\begin{itemize}
\item If $u$ is an $(\varepsilon,r,N)$-invertible of $\mathscr{K}(\ell^{p}(\Gamma)\otimes \ell^{p})+\mathbb{C}Id_{\ell^{p}(\Gamma)\otimes \ell^{p}}$, then $u$ is connected to $Id_{\ell^{p}(\Gamma)\otimes\ell^{p}}$ by a homotopy of $(\lambda\varepsilon,R,N')$-invertibles.
\item If $e$ and $f$ are $(\varepsilon,r,N)$-idempotents of $\mathscr{K}(\ell^{p}(\Gamma)\otimes\ell^{p})$ such that 
$$rank{\kappa_{0}}(e)=rank\kappa_{0}(f),$$
then $e$ and $f$ are connected by a homotopy of $(\lambda\varepsilon,R,N')$-idempotents.
\end{itemize}
\end{proposition}

\section{Applications involving $L^{p}$ coarse Baum-Connes conjecture}
In this section, $X$ will be a discrete metric space with bounded geometry and $A$ will be an $L^{p}$ operator algebra. We will present a result on the persistence approximation property of $L^{p}$ Roe algebras for $X$. This result is applied to show that if any such space is coarsely uniformly contractible and satisfies controlled-surjectivity of the $\mathscr{N}$-$L^{p}$ coarse Baum-Connes assembly map and injectivity of the $L^{p}$ coarse Baum-Connes assembly map, then the $L^{p}$ Roe algebra $B^{p}(X,A)$ has the persistence approximation property.

Assume that $\mathcal{A}=(A_{i})_{i\in\mathbb{N}}$ is any family of filtered $L^{p}$ operator algebras. For each $i\in\mathbb{N}$, there is a representation of $A_{i}$ on an $L^{p}$ space $E_{i}$. We define $E:=\bigoplus\limits_{i\in\mathbb{N}}E_{i}=\{(e_{i})_{i\in\mathbb{N}}\mid e_{i}\in E_{i}\}$ with the norm  $\Vert(e_{1},e_{2},\cdots)\Vert=\{\sum\limits_{i\in\mathbb{N}}\vert e_{i}\vert^{p}\}^{\frac{1}{p}}$. Clearly, $E$ is an $L^{p}$ space. Let $L'_{d}=\ell^{p}(Q_{d})\otimes E\otimes\ell^{p}$ be a certain $L^{p}$-$X$-module defined in \cite{ZZ21}, and let $\mathbb{C}[L'_{d},A_{i}]$ denote the algebra of all $E$-locally compact operators on $L'_{d}$ with finite propagation. For any $r>0$, we set
$$\mathcal{A}_{d,r}^{\infty}=\prod\limits_{i\in\mathbb{N}}\mathbb{C}[L'_{d},A_{i}]_{r},$$
and we define the $L^{p}$ operator algebra $\mathcal{A}_{d}^{\infty}$ as the closure of $\bigcup\limits_{r>0}\mathcal{A}_{d,r}^{\infty}$ in $\prod\limits_{i\in\mathbb{N}}B^{p}(P_{d}(X),A_{i})$.

\begin{lemma}
Let $X$ be a discrete metric space with bounded geometry, and let $\mathcal{A}=(A_{i})_{i\in\mathbb{N}}$ be a family of filtered $L^{p}$ operator algebras. Then there exist a control pair $(\lambda,h)$ independent of the family $\mathcal{A}$ and a $(\lambda,h)$-isomorphism  
$$\mathcal{G}=(G^{\varepsilon,r,N})_{0<\varepsilon<\frac{1}{20},r>0,N\geq 1}: \mathcal{K}_{*}(\mathcal{A}_{d}^{\infty})\rightarrow\prod\limits_{i \in\mathbb{N}}\mathcal{K}_{*}\big(B^{p}(P_{d}(X),A_{i})\big),$$
where 
$$G^{\varepsilon,r,N}: K^{\varepsilon,r,N}_{*}(\mathcal{A}_{d}^{\infty})\rightarrow\prod\limits_{i \in\mathbb{N}}K^{\varepsilon,r,N}_{*}\big(B^{p}(P_{d}(X),A_{i})\big)$$
is induced on the $j$-th factor by the projection $\prod\limits_{i\in\mathbb{N}}B^{p}(P_{d}(X),A_{i})\rightarrow B^{p}(P_{d}(X),A_{j})$.
\end{lemma} 

\begin{proof}
Let us first consider the even case. For $0<\varepsilon<\frac{1}{20}$, $r>0$ and $N\geq 1$, there exist a control pair $(\lambda,h)$ and a $(\lambda,h)$-controlled morphism
$$G^{\varepsilon,r,N}: K^{\varepsilon,r,N}_{*}(\mathcal{A}_{d}^{\infty})\rightarrow\prod\limits_{i \in\mathbb{N}}K^{\varepsilon,r,N}_{*}\big(B^{p}(P_{d}(X),A_{i})\big)$$
induced on the $j$-th factor by the projection $\prod\limits_{i\in\mathbb{N}}B^{p}(P_{d}(X),A_{i})\rightarrow B^{p}(P_{d}(X),A_{j})$. For any positive integer $i$ and $n$, we know that
$$M_{n}\big(\ell^{\infty}(X, A_{i}\otimes\mathscr{K}(\ell^{p}))\big)\subset\ell^{\infty}(X,A_{i})\otimes\mathscr{K}(\ell^{p}).$$
Hence, $M_{n}(B^{p}(P_{d}(X),A_{i}))\subset B^{p}(P_{d}(X),A_{i})$. 
Assume that $x$ is the element in $\prod\limits_{i\in\mathbb{N}}K^{\varepsilon,r, N}_{0}(B^{p}(P_{d}(X),A_{i}))$, then we can write  $[x]=([x_{i}])_{i\in\mathbb{N}}$ for $[x_{i}]\in K^{\varepsilon,r, N}_{0}(B^{p}(P_{d}(X),A_{i}))$. Let $(e_{i})_{i\in\mathbb{N}}$ be a family of $(\varepsilon,r,N)$-idempotents with $e_{i}$ in some $M_{n}(\widetilde{ B^{p}(P_{d}(X),A_{i})})$ such that $[x]_{\varepsilon,r,N}=[(e_{i})_{i\in\mathbb{N}}]_{\varepsilon,r,N}$, then $G^{\varepsilon,r,N}$ is $(\lambda,h)$-surjective. 

According to the item (i) of Proposition \ref{prop 2.19}, we construct the Lipschitz homotopy of $(\varepsilon,r,N)$-idempotents in larger matrix size, thus we can prove that $G^{\varepsilon,r,N}$ is $(\lambda,h)$-injective. In the odd case, we have a similar proof.
\end{proof}

\begin{lemma}\label{lemma 5.2}
Let $X$ be a discrete metric space with bounded geometry, and let $\mathcal{A}=(A_{i})_{i\in\mathbb{N}}$ be a family of filtered $L^{p}$ operator algebras, then we have a filtered isomorphism
$$\phi: B^{p}\big(P_{d}(X),\prod\limits_{i\in\mathbb{N}}A_{i}\big)\rightarrow\mathcal{A}^{\infty}_{d}.$$
\end{lemma}

\begin{proof}
By the universal property of $B^{p}(P_{d}(X),\prod\limits_{i\in\mathbb{N}}A_{i})$, there exists a filtered homomorphism
$$B^{p}(P_{d}(X),\prod\limits_{i\in\mathbb{N}}A_{i})\rightarrow\mathcal{A}^{\infty}_{d}.$$
Note that the filtered homomorphism $\phi$ maps the dense subalgebra $\mathbb{C}[L'_{d},\prod\limits_{i\in\mathbb{N}}A_{i}]$ to a dense subalgebra of $\mathcal{A}^{\infty}_{d}$, thus we can easily get that $\phi$ is surjective. It thus suffices to show that $\phi$ is injective. For every positive integer $i$, we have the inclusion $A_{i}\rightarrow\prod\limits_{i\in\mathbb{N}}A_{i}$. Hence, we have a filtered homomorphism 
$$B^{p}(P_{d}(X),A_{i})\rightarrow B^{p}(P_{d}(X),\prod\limits_{i\in\mathbb{N}}A_{i})$$
which induces a filtered homomorphism
$$\psi: {A}^{\infty}_{d}\rightarrow B^{p}(P_{d}(X),\prod\limits_{i\in\mathbb{N}}A_{i})$$
such that the composition 
$$B^{p}(P_{d}(X),\prod\limits_{i\in\mathbb{N}}A_{i})\xrightarrow{\phi}{A}^{\infty}_{d}\xrightarrow{\psi} B^{p}(P_{d}(X),\prod\limits_{i\in\mathbb{N}}A_{i})$$
is an identity map. Let $x$ be in $B^{p}(P_{d}(X),\prod\limits_{i\in\mathbb{N}}A_{i})$ such that $\phi(x)=0$ in ${A}^{\infty}_{d}$, then $x=\psi(\phi(x))=0$, thus $\phi$ is injective. This implies that $\phi$ is a filtered isomorphism.
\end{proof}

The preceding Lemma \ref{lemma 5.2} yields the following.

\begin{corollary}\label{cor 5.3}
Let $X$ be a discrete metric space with bounded geometry, and let $\mathcal{A}=(A_{i})_{i\in\mathbb{N}}$ be a family of filtered $L^{p}$ operator algebras, then there exist a control pair $(\lambda,h)$ and a $(\lambda,h)$-isomorphism
$$\mathcal{K}_{*}\big(B^{p}(P_{d}(X),\prod\limits_{i\in\mathbb{N}}A_{i})\big)\rightarrow\prod\limits_{i\in\mathbb{N}}\mathcal{K}_{*}\big(B^{p}(P_{d}(X),A_{i})\big).$$
Moreover, passing to the limit, we obtain
$$\mathcal{K}_{*}\big(B^{p}(X,\prod\limits_{i\in\mathbb{N}}A_{i})\big)\rightarrow\prod\limits_{i\in\mathbb{N}}\mathcal{K}_{*}\big(B^{p}(X,A_{i})\big).$$
\end{corollary}
\begin{definition}\cite{PAP-Oyono}
A discrete metric space $X$ is coarsely uniformly contractible, if for each $d>0$, there exists $d'>d$ such that any compact subset of $P_{d}(X)$ lies in a contractible invariant compact subset of $P_{d'}(X)$.
\end{definition}

\begin{example}\cite{PAP-Oyono}
Any discrete Gromov hyperbolic metric space is coarsely uniformly contractible.
\end{example}

\begin{definition}
Let $A$ be an $L^{p}$ operator algebra. The evaluation-at-zero homomorphism 
$$ev_{0}: B^{p}_{L}\big(P_{d}(X),A\big)\rightarrow B^{p}\big(P_{d}(X),A\big),$$
induces a homomorphism on $K$-theory
$$\mu^{d}_{A,*}=ev_{*}: K_{*}\big(B^{p}_{L}(P_{d}(X),A)\big)\rightarrow K_{*}\big(B^{p}(P_{d}(X),A)\big)\cong K_{*}\big(B^{p}(X,A)\big),$$
called an $L^{p}$ coarse assembly map. 
\end{definition}
The family of $L^{p}$ coarse assembly maps $(\mu^{d}_{A,*})_{d>0}$ gives rise to a homomorphism
$$\mu_{A,*}: \lim\limits_{d>0} K_{*}\big(B^{p}_{L}(P_{d}(X),A)\big)\rightarrow K_{*}\big(B^{p}(X,A)\big)$$
called the $L^{p}$ coarse Baum-Connes assembly map. Moreover, the $L^{p}$ coarse Baum-Connes conjecture for $X$ posits that this map $\mu_{A,*}$ is an isomorphism.

\begin{definition}
Let $A$ be an $L^{p}$ operator algebra. For $N\geq 1$, we define an $N$-$L^{p}$ coarse assembly map
$$\mu^{N,d}_{A,*}: K^{N}_{*}\big(B^{p}_{L}(P_{d}(X),A)\big)\rightarrow K^{N}_{*}\big(B^{p}(P_{d}(X),A)\big)\cong K^{N}_{*}\big(B^{p}(X,A)\big)$$
induced by the evaluation-at-zero homomorphism 
$$ev_{0}: B^{p}_{L}\big(P_{d}(X),A\big)\rightarrow B^{p}\big(P_{d}(X),A\big).$$
\end{definition}
The family of $N$-$L^{p}$ coarse assembly maps $(\mu^{N,d}_{A,*})_{d>0}$ gives rise to a homomorphism 
$$\mu^{N}_{A,*}: \lim\limits_{d>0}K^{N}_{*}\big(B^{p}_{L}(P_{d}(X),A)\big)\rightarrow  K^{N}_{*}\big(B^{p}(X,A)\big)$$
 called the $N$-$L^{p}$ coarse Baum-Connes assembly map.

\begin{remark}\label{rmk 5.8}
From the proof of Theorem 4.6 in \cite{ZZ21}, we see that if $X$ is a proper metric space with finite asymptotic dimension, then the $N$-$L^{p}$ coarse Baum-Connes assembly map for $X$ is $\omega$-surjective, and the function $\omega$ depends on the asymptotic dimension $m$, strong Lipschitz constant $C$ and Mayer-Vietoris control pair $(\lambda, h)$.
\end{remark}
The following result gives a sufficient condition for persistence approximation property to be satisfied for a class of $L^{p}$ operator algebras.

\begin{theorem}\label{th 5.6}
Let $X$ be a discrete metric space with bounded geometry, and let $A$ be an $L^{p}$ operator algebra. Assume that
\begin{itemize}
\item $X$ is coarsely uniformly contractible;
\item for any positive integer $\mathscr{N}$, there exists a non-decreasing function $\omega: [1,\infty)\rightarrow [1,\infty)$ such that the $\mathscr{N}$-$L^{p}$ coarse Baum-Connes assembly map for $X$ with coefficients in  
$$\ell^{\infty}(\mathbb{N},\mathscr{K}(\ell^{p})\otimes A)$$
is $\omega$-surjective;
\item the $L^{p}$ coarse Baum-Connes assembly map for $X$ with coefficients in $A$ is injective.
\end{itemize}
Then for any $N\geq 1$, there exists a universal constant $\lambda_{PA}\geq 1$ such that for any $\varepsilon$ in $(0,\frac{1}{20\lambda_{PA}})$ and any $r>0$, there exist $r'\geq r$ and $N'\geq N$ such that $\mathcal{PA_{*}}(B^{p}(X,A), \varepsilon,\lambda_{PA}\varepsilon,r,r',N,N')$ holds.
\end{theorem}

\begin{proof}
Let $\rho$ be as in Proposition \ref{prop 3.24}, pick $(\lambda,h)$ as in Corollary \ref{cor 5.3} and put $\lambda_{PA}=\rho(9\lambda_{N}\omega(4\lambda_{N}))$. Assume that there exists $N\geq 1$ such that this statement does not hold. Then there exist\begin{itemize}
\item $\varepsilon\in (0,\frac{1}{20\lambda_{PA}})$ and $r>0$,

\item an unbounded increasing sequence $(r_{i})_{i\in\mathbb{N}}$ bounded below by $r$,

\item an unbounded increasing sequence $(N_{i})_{i\in\mathbb{N}}$ bounded below by $N$,

\item a sequence of elements $([x_{i}])_{i\in\mathbb{N}}$ with $[x_{i}]\in K^{\varepsilon,r,N}_{*}(B^{p}(X,A))$,
\end{itemize}
such that, for each $i\in\mathbb{N}$,
$$\iota_{*}([x_{i}])=0\text{ in }K_{*}(B^{p}(X,A))$$
 and 
$$\iota^{\lambda_{PA}\varepsilon,r_{i},N_{i}}_{*}([x_{i}])\neq 0\text{ in }K^{\lambda_{PA}\varepsilon,r_{i},N_{i}}_{*}(B^{p}(X,A)).$$
Let $[x]$ be an element of $K^{\lambda_{N}\varepsilon,h_{\varepsilon,N}r,\lambda_{N}}_{*}\Big(B^{p}\big(X,\ell^{\infty}(\mathbb{N},\mathscr{K}(\ell^{p})\otimes A)\big)\Big)$ corresponding to $([x_{i}])_{i\in\mathbb{N}}$ in $\prod\limits_{i\in\mathbb{N}}K^{\varepsilon,r,N}_{*}(B^{p}(X,A))$ under the $(\lambda,h)$-isomorphism of Corollary \ref{cor 5.3}. Observe that $\iota^{\lambda_{N}}_{*}([x])$ is the element of $K^{4\lambda_{N}}_{*}\Big(B^{p}\big(X,\ell^{\infty}(\mathbb{N},\mathscr{K}(\ell^{p})\otimes A)\big)\Big)$. Assume that $$[z]\in K^{\omega(4\lambda_{N})}_{*}\Big(B^{p}_{L}\big(P_{d}(X),\ell^{\infty}(\mathbb{N},\mathscr{K}(\ell^{p})\otimes A)\big)\Big)$$ such that 
$$\mu^{\omega(4\lambda_{N}),d}_{\ell^{\infty}(\mathbb{N},\mathscr{K}(\ell^{p})\otimes A),*}([z])=\iota^{\lambda_{N}}_{*}([x])\text{ in }K^{\omega(4\lambda_{N})\cdot 4\lambda_{N}}_{*}\Big(B^{p}\big(X,\ell^{\infty}(\mathbb{N},\mathscr{K}(\ell^{p})\otimes A)\big)\Big).$$
Since the quantitative $L^{p}$ coarse assembly maps are compatible with the $\omega(4\lambda_{N})$-$L^{p}$ coarse assembly maps, we obtain that
$$\mu^{4N_{1},d}_{\ell^{\infty}(\mathbb{N}, \mathscr{K}(\ell^{p})\otimes A),*}([z]_{4N_{1}})=\iota^{N_{1}}_{*}\circ\mu^{\lambda_{N}\varepsilon,h_{\varepsilon,N}r,\omega(4\lambda_{N})}_{\ell^{\infty}(\mathbb{N}, \mathscr{K}(\ell^{p})\otimes A),*}([z]_{\omega(4\lambda_{N})}),$$
where $N_{1}=\max\{\omega(4\lambda_{N})\cdot\lambda_{N}, 9\omega(4\lambda_{N})\}$. However,
according to Proposition \ref{prop 3.24}, there exists $R\geq h_{\varepsilon,N}r$ such that
$$\iota^{\lambda_{PA}\varepsilon,R,33N_{1}}_{*}([x])=\mu^{\lambda_{PA}\varepsilon, R, 33N_{1},d}_{\ell^{\infty}(\mathbb{N}, \mathscr{K}(\ell^{p})\otimes A),*}([z]_{33N_{1}}).$$
By Proposition \ref{prop 3}, we have an isomorphism
$$K_{*}\Big(B^{p}_{L}\big(P_{d}(X),\ell^{\infty}(\mathbb{N},\mathscr{K}(\ell^{p})\otimes A)\big)\Big)\cong\prod\limits_{i\in\mathbb{N}}K_{*}\big(B^{p}_{L}(P_{d}(X),A)\big).$$
Let $([z_{i}])_{i\in\mathbb{N}}$ be the element of $\prod\limits_{i\in\mathbb{N}}K_{*}(B^{p}_{L}(P_{d}(X),A))$ corresponding to $[z]$ under this identification. 
Using the compatibility of the quantitative $L^{p}$ assembly maps with the usual ones, we obtain by naturality that $\mu^{d}_{A,*}([z_{i}])=0$, for each $i\in\mathbb{N}$. Since $X$ is coarsely uniformly contractible and $\mu_{A,*}$ is injective, we deduce that there exists $d'\geq d$ such that
 $$i_{d,d',*}([z])=0\text{ in }K_{*}\Big(B^{p}_{L}\big(P_{d'}(X),\ell^{\infty}(\mathbb{N},\mathscr{K}(\ell^{p})\otimes A)\big)\Big).$$
Let $(p_{t})_{t\in [0,1]}$ be a homotopy of idempotents (resp. invertibles) in $M_{n}(\widetilde{B})$ between $i_{d,d',*}([z])$ and $0$, then $P:=(p_{t})$ is an idempotent (resp. invertible) element in $C([0,1], M_{n}(\widetilde{B}))$, where $B=B^{p}_{L}\big(P_{d'}(X),\ell^{\infty}(\mathbb{N}, \mathscr{K}(\ell^{p})\otimes A)\big)$. Put $N'=\max\{33N_{1},\Vert P\Vert\}$. Since
 $$\mu^{\lambda_{PA}\varepsilon,R,N',d}_{\ell^{\infty}(\mathbb{N},\mathscr{K}(\ell^{p})\otimes A),*}([z])=\mu^{\lambda_{PA}\varepsilon,R,N',d'}_{\ell^{\infty}(\mathbb{N},\mathscr{K}(\ell^{p})\otimes A),*}\circ i_{d,d',*}([z]),$$
then
$$\iota^{\lambda_{PA}\varepsilon,R,N'}_{*}([x])=0\text{ in }K^{\lambda_{PA}\varepsilon,R,N'}_{*}\Big(B^{p}\big(X,\ell^{\infty}(\mathbb{N},\mathscr{K}(\ell^{p})\otimes A)\big)\Big).$$
By naturality, we see that $\iota^{\lambda_{PA}\varepsilon,R,N'}_{*}([x_{i}])=0$ in $K^{\lambda_{PA}\varepsilon,R,N'}_{*}\big(B^{p}_{L}(X,A)\big)$ for all integers $i$. Picking an integer $i$ such that $r_{i}\geq R$ and $N_{i}\geq N'$, we have
$$\iota^{\lambda_{PA}\varepsilon,r_{i},N_{i}}_{*}([x_{i}])=0,$$
which contradicts our assumption.
\end{proof}

\begin{theorem}\cite{ZZ21}\label{th 5.7}
For any $p\in[1,\infty)$, the $L^p$ coarse Baum–Connes conjecture holds for
proper metric spaces with finite asymptotic dimension.
\end{theorem}

Since hyperbolic metric spaces have finite asymptotic dimensions, and combining this with Remark \ref{rmk 5.8}, Theorem \ref{th 5.6} and Theorem \ref{th 5.7}, we have the following result.
\begin{corollary}
For any $N\geq 1$, there exists a universal constant $\lambda_{PA}\geq 1$ such that for any discrete Gromov hyperbolic metric space $X$, the following holds: for any $\varepsilon$ in $(0,\frac{1}{20\lambda_{PA}})$ and any $r>0$, there exist $r'\geq r$ and $N'\geq N$ such that $\mathcal{PA}_{*}(B^{p}(X,A),\varepsilon,\lambda_{PA}\varepsilon,r,r',N,N')$ holds for any $L^{p}$ operator algebra $A$.

\end{corollary}

\section*{Acknowledgments}
HW is supported by the grant NSFC 12271165 and in part by Science and Technology Commission of Shanghai Municipality (No. 22DZ2229014), and JZ is supported by the grants NSFC 12171156, 12271165. We thank Eusebio Gardella for suggesting us many useful references on $L^{p}$ operator algebras.



\begin{thebibliography}{STF}

\bibitem{Blackadar}Blackadar, B., 
{$K$-­Theory for Operator Algebras}, 
Second edition, Mathematical Sciences Research Institute Publications, {\bf5}, Cambridge University Press, Cambridge, 1998.

\bibitem{Banach-Chung} Chung, Y. C.,
{Quantitative $K$-­theory for Banach algebras}, 
\emph{J. Funct. Anal.}, {\bf274}(1), 2018, 278--340.

\bibitem{DC-Chung} Chung, Y. C.,
{Dynamical complexity and $K$-theory of $L^p$ operator crossed products}, 
\emph{J. Topol. Anal.}, {\bf13}(3), 2021, 809--841.

\bibitem{Rigidity}Chung, Y. C. and Li, K.,  
{Rigidity of $\ell^p$ Roe-type algebras}, 
\emph{Bull. Lond. Math. Soc.},  {\bf 50}(6), 2018, 1056--1070.

\bibitem{uniform}Chung, Y. C. and Li, K.,
{Structure and $K$-theory of $\ell^p$ uniform Roe algebras}, 
\emph{J. Noncommut. Geom.}, {\bf 15}(2), 2021, 581--614.

\bibitem{daws}Daws, M.,
{$p$-operator spaces and Figà-Talamanca-Herz algebras}, 
\emph{J. Operator Theory}, {\bf 63}(1), 2010, 47--83.

\bibitem{Lp algebra-Gardella} Gardella, E.,
{A modern look at algebras of operators on $L^p$-spaces},
\emph{Expo. Math.}, {\bf 39}(3), 2021, 420--453.

\bibitem{groupoids}Gardella, E. and Lupini, M.,
{Representations of étale groupoids on $L^{p}$-spaces},
\emph{Adv. Math.}, {\bf 318}, 2017, 233--278.

\bibitem{Group algebras}Gardella, E. and Thiel, H., 
{Group algebras acting on $L^{p}$-spaces},
\emph{J. Fourier Anal. Appl.}, {\bf 21}(6), 2015, 1310--1343.

\bibitem{Quotients} Gardella, E. and Thiel, H.,
{Quotients of Banach algebras acting on $L^{p}$-spaces}, 
\emph{Adv. Math.}, {\bf 296}, 2016, 85--92.

\bibitem{$p$-convolution}Gardella, E. and Thiel, H.,
{Representations of $p$-convolution algebras on $L^{q}$-spaces},
\emph{Trans. Amer. Math. Soc.}, {\bf 371}(3), 2019, 2207--2236.

\bibitem{quasi-local}Li, K., Wang, Z. and Zhang, J.,
{A quasi-local characterisation of $L^p$-Roe algebras}, 
\emph{J. Math. Anal. Appl.}, {\bf 474}(2), 2019, 1213--1237.

\bibitem{Oyono-2015} Oyono-Oyono, H. and Yu, G.,
{On quantitative operator $K$-theory},
\emph{Ann. Inst. Fourier (Grenoble)}, {\bf 65}(2), 2015, 605--674.

\bibitem{PAP-Oyono} Oyono-Oyono, H. and Yu, G., 
{Persistence approximation property and controlled operator $K$-theory}, 
\emph{M$\ddot{\mathrm{u}}$nster J. Math.}, {\bf 10}(2), 2017, 201--268.

\bibitem{Phillips} Phillips, N., 
{Crossed products of $L^p$ operator algebras and the $K$-theory of Cuntz algebras on $L^p$
spaces}, 2013, arXiv:1309.6406.

\bibitem{Pisier} Pisier, G.,
{Completely bounded maps between sets of Banach space operators}, 
\emph{Indiana Univ. Math. J.}, {\bf 39}(1), 1990, 249--277.

\bibitem{WZ-PAP} Wang, Q. and  Wang, Z.,
{Persistence approximation property for maximal Roe algebras},
\emph{Chinese Ann. Math. Ser. B}, {\bf 41}(1), 2020, 1--26.

\bibitem{WZ23}Wang, Z. and Zhu, S., 
{On the Takai duality for $L^{p}$ operator crossed products},
\emph{Math. Z.}, {\bf 304}(54), 2023. https://doi.org/10.1007/s00209-023-03316-4.

\bibitem{Higher-Yu}Willett, R. and Yu, G., {Higher Index Theory},  {Cambridge Studies in Advanced Mathematics}, {\bf 189}, Cambridge University Press, Cambridge, 2020.

\bibitem{1998-Yu} Yu, G.,
{The Novikov conjecture for groups with finite asymptotic dimension}, 
\emph{Ann. of Math. (2)}, {\bf 147}(2), 1998, 325--355.

\bibitem{ZZ21}Zhang, J. and Zhou, D.,
{$L^{p}$ coarse Baum–Connes conjecture and $K$-theory for $L^{p}$ Roe algebras},
\emph{J. Noncommut. Geom.}, {\bf 15}(4), 2021, 1285--1322.
\end{thebibliography}
\end{document}